\documentclass[11pt]{article}
\usepackage[utf8]{inputenc}
\usepackage[T1]{fontenc}
\usepackage{amsmath,textcomp,amssymb,geometry,graphicx,enumerate}
\usepackage{amsthm}
\usepackage{thmtools}
\theoremstyle{plain}
\newtheorem{proposition}{Proposition}[section]
\newtheorem{theorem}{Theorem}[section]
\newtheorem{lemma}{Lemma}[section]
\newtheorem{problem}{Problem}[section]
\theoremstyle{remark}
\newtheorem{remark}{Remark}[section]
\theoremstyle{definition}
\newtheorem{definition}{Definition}[section]

\usepackage{algorithm} 
\usepackage[noend]{algpseudocode} 
\usepackage{hyperref}
\usepackage{stmaryrd}
\hypersetup{
    colorlinks=true,
    linkcolor=blue,
    filecolor=magenta,
    urlcolor=blue,
}
\title{Branched $\alpha$-combinatorial Ricci flows on closed surfaces with Euler characteristic $\chi\le 0$}
\author{Wenjun Li, Rongyuan Liu, Guohao Chen, Aijin Lin}
\date{}

\begin{document}
\maketitle
\begin{abstract}
In this paper we introduce the branched $\alpha$-flows on closed surfaces with Euler characteristic \(\chi \leq 0\). Based on the strict convexity of  the branched $\alpha$-potentials, we establish the long time existence and convergence of the solutions to the branched $\alpha$-flows, which generalizes Ge and Xu's main results \cite{2015,2015A} on the $\alpha$-flows. In addtion, we study the prescribed curvature problems under the relaxed precondition $\chi(M)\in \mathbb{Z}$ via alternative $\alpha$-flows, establishing admissibility conditions for prescribed curvatures and their exponential convergence to target metrics.

\end{abstract}

\vspace{12pt}
\let\thefootnote\relax\footnotetext {}

\section{Introduction}
\hspace{14pt}
After Thurston \cite{1980The} proposed using circle packings to approximate conformal mappings, Beardon and Stephenson \cite{1990The} suggested that circle packings could be regarded as discrete analogues of analytic functions, and similar conclusions could be derived from classical results in complex analysis. The branched circle packings were introduced by Bowers and Stephenson \cite{bowers}, who used it to prove the branched Koebe-Andreev-Thurston theorem. Subsequently, Dubejko \cite{1995Branched} continued the study of branched circle packings and established necessary and sufficient conditions for the existence of branched circle packings on surfaces.

On the other hand, Chow and Luo \cite{2002Combinatorial} introduced the combinatorial Ricci flows analogous to Hamilton's Ricci flow \cite{Hamilton1982Three} to deform circle packings, and provided an alternative proof of the Koebe-Andreev-Thurston theorem. Subsequently, Lan and Dai \cite{2007Variational} combined the branch structures with Chow-Luo's combinatorial Ricci flows and introduced the branched combinatorial Ricci flows. They successfully extended Chow-Luo's results on the combinatorial Ricci flows. Gao-Lin \cite{gl1, gl2} introduced branched combinatorial Calabi flows and generalized Lan-Dai's work on branched combinatorial Ricci flows from $p=2$ to any $p>1$. Ge and Xu \cite{2015,2015A} introduced $\alpha$-curvatures and $\alpha$-flows and studied the combinatorial Yamabe problem on triangulated surfaces.

 In this paper, we shall combine $\alpha$-flows on closed surfaces with branch structures, and introduce branched $\alpha$-flows as follows (see Section \ref{sec 3} for details). Given a weighted triangulated closed surface $(M,\mathcal{T},\Phi,H)$ with a circle packing metric $r$ and a non-negative real number $\alpha$, a branched $\alpha$-flow in Euclidean background geometry $\mathbb{E}^2$ is defined to be
\begin{equation}\label{flow E}
\frac{\mathrm{d}u_i^{\mathbb{E}}}{\mathrm{d}t} = s_{\alpha}^{\mathbb{E}}r_{i}^{\alpha}-(K_i+2\pi\beta_i), \forall v_i\in V.
\end{equation}
When $\chi(M)\le -1$, a branched $\alpha$-flow in hyperbolic background geometry $\mathbb{H}^2$ is defined to be
\begin{equation}\label{flow H}
\frac{\mathrm{d}u_i^{\mathbb{H}}}{\mathrm{d}t} = s_{\alpha}^{\mathbb{H}}\tanh^{\alpha}\frac{r_i}{2}-(K_i+2\pi\beta_i), \forall v_i\in V.
\end{equation}
Our main result is to establish the long time existence and convergence of the solution $u(t)$ to the branched $\alpha$-flows \eqref{flow E} \eqref{flow H} as follows.
\begin{theorem}\label{main result}
For any initial metric $r(0)$, the solution to \eqref{flow E} (\eqref{flow H} resp.) exists for all time $t\in \mathbb{R}$ and converges exponentially fast to the constant $\alpha$-metric $r_H^{\mathbb{E}}$ ($r_H^{\mathbb{H}}$ resp.) in $\mathbb{E}^2$ ($\mathbb{H}^2$ resp.).
\end{theorem}
Theorem \ref{main result} generalizes Ge and Xu's main results \cite{2015,2015A} on the $\alpha$-flows. In addition, we introduced the branched $\alpha$-potential and establish its strict convexity, which may be of some independent interest. Finally, we study the prescribed curvature problems under the relaxed precondition $\chi(M)\in \mathbb{Z}$ via alternative $\alpha$-flows, establishing admissibility conditions for prescribed curvatures and their exponential convergence to target metrics.\par
This paper is organized as follows. In Section \ref{sec 2}, we give some preliminaries. In Section $\ref{sec 3}$, we introduce the branched $\alpha$-flows and prove the long time existence of the solution to \eqref{flow E} (\eqref{flow H} resp.). In Section \ref{section 4} , we introduce the branched $\alpha$-potential and establish its strict convexity. In Section \ref{section 5}, we establish the exponential convergence of the solution to \eqref{flow E} (\eqref{flow H} resp.). In Section \ref{sec 6}, under the relaxed precondition $\chi(M)\in \mathbb{Z}$, two additional branched $\alpha$-flows are introduced to address the prescribed curvature problems. In Section \ref{sec 7}, we present several unsolved problems.

\section{Preliminaries}\label{sec 2}
\hspace{14pt}
Suppose $M$ is a closed connected surface equipped with a triangulation $\mathcal{T}=\{V, E, F\}$, where $V$, $E$, and $F$ denote the sets of vertices, edges, and faces, respectively. A \emph{weight} on the triangulation is defined to be a
function $\Phi : E\rightarrow[0, \pi/2]$ that assigns a weight $\Phi_{ij}$ to each edge $e_{ij}$. The triple $(M, \mathcal{T}, \Phi)$ is called a weighted triangulation of $M$. Denote by $N = |V|$ the number of vertices. We write $v_i \sim v_j$ to indicate that vertices $v_i$ and $v_j$ are adjacent, where $1 \leq i,j \leq N $. Throughout this paper, any function $f: V \to \mathbb{R}$ is treated as a column vector $f = (f_1,\ldots,f_N)^T$ in $\mathbb{R}^N$, where $f_i$ represents the value of $f$ at vertex $v_i$. Additionally, we use $C(V)$ to denote the set of all functions defined on $V$.

Let $P$ be a collection of circles on the surface $M$. Then we say that $P$ is a
\emph{circle packing} (see \cite{bowers, gl2} for details) realizing a weighted triangulation $(\mathcal{T}, \Phi)$ if
\begin{itemize}
    \item {there exists a $1-$to$-1$ correspondence between the vertices $v \in V$ and the
    circles $C_{p}(v)$ of $P$ such that
    circles $C_{p}(u)$ and $C_{p}(\omega)$ intersect at an overlap angle $\Phi(e_{u\omega})$,
    where $e_{u\omega}$ is the edge joining $u$ and $\omega$.}
    \item {$P$ is orientation preserving, i.e. if $v_1$, $v_2$, $v_3$ are the vertices of a
    face in $\mathcal{T}$ taken in the positive order
    (with respect to the orientation of $\mathcal{T}$) then $C_p(v_1)$, $C_p(v_2)$, $C_p(v_3)$
    form a positively oriented triple of circles in the surface.}
\end{itemize}
In order to better understand the branched circle packings, we introduce a radius function $r = (r_1,\ldots,r_N)^T: V\to(0,+\infty)$ called a \emph{circle packing metric}. Let $\theta_{i}^{ijk}$ denote the inner cone angle of the triangle $\Delta_{ijk}$ at $v_i$. The \emph{discrete Gauss curvature} at the vertex $v_i$ is defined as
\begin{equation}\label{def K}
    K_{i}=2\pi-\sum_{f_{ijk}\in F}\theta_{i}^{ijk},
\end{equation}
where the sum is taken over all the triangles having \(v_i\) as one of their vertices. A vertex $v_i$ is called a \emph{branch point} of $P$ with order
$n - 1$, if its angle sum $\theta_i=\sum_{f_{ijk}\in F}\theta_{i}^{ijk}$ equals to $2n\pi$, $n\geq2$.\par

Given $(M,\mathcal{T},\Phi)$, we can assign each edge $e_{ij}$ a length as
$$l_{ij} =\sqrt{r_{i}^{2}+r_{j}^{2}+2r_{i}r_{j}\cos(\Phi_{ij})},$$
in Euclidean background geometry and
$$l_{ij} =\cosh^{-1}(\cosh r_{i}\cosh r_{j}+\sinh r_{i}\sinh r_{j}\cos(\Phi_{ij})),$$ in hyperbolic background geometry.
As a consequence, each face $\Delta_{ijk}\in F$ is isometric to a Euclidean triangle (hyperbolic resp.). Specifically,
each face $\Delta_{ijk}\in F$ is a Euclidean triangle (hyperbolic resp.) with edge lengths $l_{ij}, l_{jk}, l_{ik}$
because $l_{ij}, l_{jk}, l_{ik}$ satisfy triangle inequalities (Lemma 13.7.2 in \cite{1980The}). Furthermore,
the triangulated surface $(X, \mathcal{T})$ is composed by gluing Euclidean triangles coherently (hyperbolic resp.).

\begin{definition}
Given a weighted triangulated surface $(M,\mathcal{T},\Phi)$, if branched points $\{b_j\in V\mid j = 1,\ldots,m\}$ satisfy
\[
\Theta(b_j)=(1 + \beta_j)2\pi,
\]
and for $v_i\in V\setminus\{b_1,\ldots,b_m\}$, the condition
\[
\Theta(v_i)=2\pi
\]
holds, then we call $b_j$ a branch point of order $\beta_j\in\mathbb{Z}_{> 0}$ and denote $\beta=(\beta_1,\beta_2,\ldots,\beta_N)^T$. If $b_1,\ldots,b_m$ are all the branch points of the circle packing $P$ with orders $\beta_1,\ldots,\beta_m$ respectively, then
\[
br(P)=\{(b_1,\beta_1),(b_2,\beta_2),\ldots,(b_m,\beta_m)\}
\]
is called as the branch set of $P$, and $P$ is referred to as a branched circle packing with branch set $br(P)$ that realizes $(M,\mathcal{T},\Phi)$.
\end{definition}
\begin{remark}
    If $v_i$ is not a branched point, we assign an order $\beta_i$ to it and set $\beta_i = 0$.
\end{remark}
\begin{definition}\label{def branch}
A branch set $br(P)=\{(b_1,\beta_1),(b_2,\beta_2),\ldots,(b_m,\beta_m)\}\subset V\times m$ of a circle packing $P$ is called a branch structure if each simple closed path $\varGamma = \{e_1,e_2,\ldots,e_n\}$ in $\mathcal{T}$, satisfies \[
\sum_{j = 1}^{n}[\pi - \Phi(e_j)]>2(l(\varGamma)+1)\pi.
\]
Here $l(\varGamma)=\sum_{i = 1}^{m}\delta_i(\varGamma)\beta_i$ and $\delta_i = 1$ if $v_i$ is enclosed by $\varGamma$  or else $\delta_i = 0$.
\end{definition}

Dubejko's work\cite{1995Branched} showed that there exists a branched circle packing $P$ realizing the weighted triangulation $(M, \mathcal{T}, \Phi)$ if and only if $br(P)$ is a branch structure for $(M, \mathcal{T}, \Phi)$. In the following, we use the four-tuple $(M, \mathcal{T},\Phi,P)$ to represent $(M, \mathcal{T}, \Phi)$ with a branch structure $br(P)$, and we say it as a branched weighted triangulated surface.

Chow and Luo \cite{2002Combinatorial} introduced the
combinatorial Ricci flows on surfaces and proved that the solutions to the flow equations exist for
all time and converge exponentially fast to the Thurston's circle packings \cite{1980The}. Since then, various discrete curvature flows were introduced and studied. Specially, motivated by Chow-Luo's pioneering work, Ge and Xu \cite{2015} introduced a kind of new discrete curvatures, namely $\alpha$-curvature $$R_{\alpha,i} = \frac{K_i}{r_i^{\alpha}},$$
which can be used to approximate the Gauss
curvature on surfaces. Furthermore, Ge and Xu defined $\alpha$-flow in $\mathbb{E}^2$ as
\begin{equation}\label{ge original flow}
\frac{\mathrm{d}r_i}{\mathrm{d}t}=(R_{\alpha,av}-R_{\alpha,i})r_i,
\end{equation}
where $$R_{\alpha,av} = 2\pi\chi(M)/\|r\|_{\alpha}^{\alpha}, \|r\|_{\alpha}=(\sum_{i = 1}^{n}r_i^{\alpha})^{\frac{1}{\alpha}}.$$
For convenience, we choose the coordinate transformation $u_i=\ln r_i$. The map $u = u(r)$ is a diffemorphism between $\mathbb{R}^N_{>0}$ and $\mathbb{R}^N$. It is noteworthy that it remains unknown whether (\ref{ge original flow}) exists for all time. Ge and Xu introduced another $\alpha$-flow as
\begin{equation}\label{ge-xu's flow}
\frac{\mathrm{d}u_i}{\mathrm{d}t} = R_{\alpha,av}r_i^{\alpha}-K_i,
\end{equation}
and further proved the results similar to Chow-Luo's combinatorial Ricci flows.

\section{Branched $\alpha$-Flows}\label{sec 3}
\hspace{14pt}
In Ge and Xu's work \cite{2015}, $\alpha$ is considered an arbitrary real number. However, this choice leads to the loss of geometric meaning when $\alpha < 0$. To address this, we restrict $\alpha$ to non-negative real numbers throughout this paper.
In this section, we introduce two types of branched $\alpha$-curvatures and their corresponding $\alpha$-flows in $\mathbb{E}^2$ and $\mathbb{H}^2$.

Chow and Luo \cite{2002Combinatorial} considered the circle packing metric $r_i$ as the discrete analogue of the smooth Riemannian metric tensor $g$, and the discrete Gauss curvature $K_i$ as that of the smooth Riemannian curvature. Ge and Xu \cite{2015} regarded $r_i^2$ as a suitable analogue of the Riemannian metric in $\mathbb{E}^2$, and $K_i/r_i^2$ as a suitable analogue of the smooth Gaussian curvature. Furthermore, they generalized these to the $\alpha$-th power. Inspired by Ge-Xu's work, we introduced the branched $\alpha$-curvature for circle packing metrics.
\begin{definition}
Given $(M,\mathcal{T},\Phi,H)$ with a circle packing metric $r\colon V\to(0,+\infty)$ and a non-negative real number $\alpha$, a branched $\alpha$-curvature at a vertex $v_i$ in $\mathbb{E}^2$ is defined to be
\begin{equation}\label{curvature E}
    B_{\alpha,i}^{\mathbb{E}}(r_i) = \frac{K_i+2\pi\beta_i}{r_i^\alpha}.
\end{equation}
When the Euler characteristic of $M$ is negative, a branched $\alpha$-curvature at a vertex $v_i$ in $\mathbb{H}^2$ is defined to be
\begin{equation}\label{curvature H}
    B_{\alpha,i}^{\mathbb{H}}(r_i) = \frac{K_i+2\pi\beta_i}{\tanh^\alpha \frac{r_i}{2}}.
\end{equation}
\end{definition}
\begin{remark}
     In the case $\alpha=2$, $A_i^{\mathbb{E}} = \pi r_i^2$ is the area of the disk at $v_i$ in $\mathbb{E}^2$. However, although $A_i^{\mathbb{H}} = 4\pi\sinh^2 \frac{r_i}{2}$ represents the area of a circle with radius $r_i$ in $\mathbb{H}^2$, we do not take $\sinh^{\alpha} \frac{r_i}{2}$ as the denominator in (\ref{curvature H}), which will be explained in the Remark \ref{reason not sinh}. In \cite{2015A}, Ge proved that $r_i^2$ and $\sinh^2 \frac{r_i}{2}$ serve as suitable discrete analogues to the smooth Riemannian metric tensor $g$ in $\mathbb{E}^2$ and $\mathbb{H}^2$, respectively. Here, $\tanh^2 \frac{r_i}{2}$ can also be regarded as a good analogy for $g$, because as the triangulation becomes finer, we have $\tanh \frac{r_i}{2}\sim\sinh \frac{r_i}{2}$ when $r_i\rightarrow0$.
\end{remark}

Ge and Xu \cite{2015} proposed the $\alpha$-flows in $\mathbb{E}^2$. On the other hand, Lan and Dai \cite{2007Variational} introduced the branched combinatorial Ricci flows in both $\mathbb{E}^2$ and $\mathbb{H}^2$. Inspired by their work, we introduced branched $\alpha$-flows as follows.
\begin{definition}
Given $(M,\mathcal{T},\Phi,H)$ with a circle packing metric $r$ and a non-negative real number $\alpha$, a branched $\alpha$-flow in $\mathbb{E}^2$ is defined to be
\begin{equation}\label{main flow E}
\frac{\mathrm{d}u_i^{\mathbb{E}}}{\mathrm{d}t} = s_{\alpha}^{\mathbb{E}}r_{i}^{\alpha}-(K_i+2\pi\beta_i), \forall v_i\in V,
\end{equation}
where $u_i^{\mathbb{E}} = \ln r_i$, and $s_{\alpha}^{\mathbb{E}}=2\pi\chi(M)/\lVert r\rVert_{\alpha}^{\alpha}$.

When $\chi(M)\le -1$, a branched $\alpha$-flow in $\mathbb{H}^2$ is defined to be
\begin{equation}\label{main flow H}
\frac{\mathrm{d}u_i^{\mathbb{H}}}{\mathrm{d}t} = s_{\alpha}^{\mathbb{H}}\tanh^{\alpha}\frac{r_i}{2}-(K_i+2\pi\beta_i), \forall v_i\in V,
\end{equation}
where $u_i^{\mathbb{H}} = \ln \operatorname{tanh}(r_i/2)$, and $s_{\alpha}^{\mathbb{H}}=2\pi\chi(M)/\lVert \tanh\frac{r}{2}\rVert_{\alpha}^{\alpha}$.
\end{definition}
\begin{remark}
    Lan and Dai’s branched combinatorial Ricci flow is precisely the \eqref{main flow E} with $\alpha = 0$. Additionally, Ge and Xu’s $\alpha$-flow
    \eqref{ge-xu's flow} in $\mathbb{E}^2$  is precisely the \eqref{main flow E} with $\beta = 0$. Hence, the branched $\alpha$-flow \eqref{main flow E} is a common generalization of Lan and Dai’s branched combinatorial Ricci flow and Ge and Xu's $\alpha$-flow.
\end{remark}

\begin{definition}
A circle packing metric \( r_H^{\mathbb{E}}\) is called a constant branched \(\alpha\)-curvature circle packing metric(constant branched \(\alpha\)-metric for short) in $\mathbb{E}^2$ if for any vertex $v_i\in V$ it satisfies:
\begin{equation}
B_{\alpha,i}^{\mathbb{E}}(r_H^{\mathbb{E}})= s_{\alpha}^{\mathbb{E}},
\end{equation}
which is equivalent to the explicit form:
\[
s_{\alpha}^{\mathbb{E}}{r_H^{\mathbb{E}}}^\alpha= K_i+2\pi\beta_i.
\]
Similarly, a circle packing metric \(r_H^{\mathbb{H}}\) is called a constant branched \(\alpha\)-curvature circle packing metric(constant branched \(\alpha\)-metric for short) in $\mathbb{H}^2$ if any vertex $v_i\in V$ it satisfies:
\begin{equation}
B_{\alpha,i}^{\mathbb{H}}(r_H^{\mathbb{H}})= s_{\alpha}^{\mathbb{H}},
\end{equation}
with the equivalent explicit form:
\[
s_{\alpha}^{\mathbb{H}}\tanh^\alpha\frac{r_H^{\mathbb{H}}}{2}= K_i+2\pi\beta_i.
\]
\end{definition}

\begin{remark}
It is straightforward to see that if \(\eqref{main flow E}\) and \(\eqref{main flow H}\) converge, then the constant branched $\alpha$-metrics exist in \(\mathbb{E}^2\) and \(\mathbb{H}^2\) respectively. Regardless of the property of convergence or divergence of the branched \(\alpha\)-flows, we will directly prove the existence and uniqueness of constant branched $\alpha$-metrics \(r_H^{\mathbb{E}}\) in \(U = \{u \in \mathbb{R}^N \mid \sum_i u_i = 0\} \subset \mathbb{E}^2\) and \(r_H^{\mathbb{H}}\) in \(\mathbb{H}^2\) in Theorems \(\ref{thm critical point E}\) and \(\ref{thm critical point H}\), relying solely on topological-combinatorial conditions.
\end{remark}
\begin{lemma}[\cite{2002Combinatorial}, Lemma 3.5]\label{Lemma eps}
    For any \(\epsilon>0\), there exists a number \(l\) such that when \(r_i > l\), the inner angle \(\theta_i^{jk}\) in the triangle \(\triangle v_iv_jv_k\) is smaller than \(\epsilon\) in $\mathbb{H}^2$.
\end{lemma}
By Lemma \ref{Lemma eps} we have
\begin{proposition}[Long-Time Existence]\label{long time existance}
The solutions to \eqref{main flow E} and \eqref{main flow H} exist for all time $t\in \mathbb{R}$.
\end{proposition}
\begin{proof}
Since $(2-d)\pi +2\pi\beta_i< K_i+2\pi\beta_i < 2\pi+2\pi\beta_i$, where $d=\max_{1\le i\le N}d_i$, and $d_i$ denotes the degree of vertex $i$. Note that
$$|s_\alpha^{\mathbb{E}} r_i^\alpha|\leq |s_\alpha^{\mathbb{E}}|\lVert r\rVert_{\alpha}^{\alpha}=2\pi|\chi(M)|,$$
$$|s_\alpha^{\mathbb{H}} \tanh^{\alpha} \frac{r_i}{2}|\leq |s_\alpha^{\mathbb{H}}|\lVert \tanh\frac{r}{2}\rVert_{\alpha}^{\alpha}=2\pi|\chi(M)|,$$
thus we have
\[(2-d)\pi +2\pi\beta_i-2\pi|\chi(M)|< K_i+2\pi\beta_i-s_\alpha^{\mathbb{E}} r_i^\alpha < 2\pi+2\pi\beta_i+2\pi|\chi(M)|\]
\[(2-d)\pi +2\pi\beta_i-2\pi|\chi(M)|< \ K_i+2\pi\beta_i-s_\alpha^{\mathbb{H}} \tanh^{\alpha} \frac{r_i}{2} < 2\pi+2\pi\beta_i+2\pi|\chi(M)|.\]
Set $a_1=(2-d)\pi +2\pi\beta_i-2\pi|\chi(M)|, a_2=2\pi+2\pi\beta_i+2\pi|\chi(M)|$. Notice \eqref{main flow E} can be rewritten as

\[\frac{\mathrm{d}r_i}{\mathrm{d}t} = [s_{\alpha}^{\mathbb{E}}r_{i}^{\alpha}-(K_i+2\pi\beta_i)]r_i,
\]
thus there exist two positive constants $c_1, c_2$ such that
\[
c_1e^{a_1t}<r_i(t)<c_2e^{a_2t}, \forall v_i\in V.
\]

Therefore, $r_i(t)$ is bounded above and bounded below away from $0$ in $\mathbb{E}^2$ in finite time. Then by the extension theorem in ODE theory, we get the long time existence of the solutions to \eqref{main flow E}. Similarly, \eqref{main flow H} can be rewritten as
\begin{equation}\label{ineq H}
\frac{\mathrm{d}r_i}{\mathrm{d}t} = [s_{\alpha}^{\mathbb{H}}\tanh^{\alpha}\frac{r_i}{2}-(K_i+2\pi\beta_i)]\sinh r_i.
\end{equation}

The inequality above also implies that $r_i$ is bounded below in $\mathbb{H}^2$ when $t$ is bounded.
On the other hand, if $r_i(t)$ is not bounded above when $t$ is bounded, then $\varlimsup_{t\rightarrow\tau}r_i(t)=+\infty$, where $\tau$ is the maximal existence time of the flow \eqref{main flow H}. For the vertex $v_i$, by Lemma \ref{Lemma eps}, we can choose $l > r_i(0)$ large enough so that, whenever $r_i>l$, the inner angle $\theta_i^{jk}$ is smaller than $\frac{\pi}{d_i}$. Thus, $K_i = 2\pi-\sum\theta_i^{jk}>\pi>0$. Choose a time $t_0$ such that $r_i(t_0)>l$, this can be done since \(\varlimsup_{t\rightarrow\tau}r_i(t)=+\infty\). Denote $p=\inf\{t_1 < t_0|r_i(t)>l,\forall t_1<t<\tau\}$, then $r_i(p)=l$. Since $K_i>0$ when $p\leq t\leq t_0$, then $r_i^{\prime}(t)=[s_{\alpha}^{\mathbb{H}}\tanh^{\alpha}\frac{r_i}{2}-(K_i+2\pi\beta_i)]\sinh r_i<0$ by \eqref{ineq H}. Hence $r_i(t)\leq r_i(p)=l$, which contradicts $r_i(t_0)>l$. Thus, in $\mathbb{H}^2$, $r_i$ is also bounded from above in $\mathbb{H}^2$ when $t$ is bounded. By the extension theorem, we obtain the long time existence of the solutions to \eqref{main flow H}.
This complete the proof.
\end{proof}

\section{Branched $\alpha$-Potentials}\label{section 4}
\hspace{14pt}
In this section, we introduce the branched $\alpha$-potentials in $\mathbb{E}^2$ and $\mathbb{H}^2$, which are key to prove the convergence of the solutions to the branched combinatorial Ricci flows \eqref{main flow E} \eqref{main flow H}.
\begin{definition}
We define branched $\alpha$-order combinatorial Ricci potentials (branched $\alpha$-potentials for short) as
\begin{equation}\label{main ricci potential E}
F^{\mathbb{E}}(u^{\mathbb{E}})=\int_{u_*^{\mathbb{E}}}^{u^{\mathbb{E}}}\sum_{i = 1}^{N}((K_i+2\pi\beta_i)-s_{\alpha}^{\mathbb{E}}r_{i}^{\alpha})\mathrm{d}u^{\mathbb{E}}_{i},u^{\mathbb{E}}\in \mathbb{R}^N
\end{equation}
in $\mathbb{E}^2$, and
\begin{equation}\label{main ricci potential H}
    F^{\mathbb{H}}(u^{\mathbb{H}})=\int_{u_*^{\mathbb{H}}}^{u^{\mathbb{H}}}\sum_{i = 1}^{N}((K_i+2\pi\beta_i)-s_{\alpha}^{\mathbb{H}}\tanh^\alpha\frac{r_i}{2})\mathrm{d}u_{i}^{\mathbb{H}},u^{\mathbb{H}}\in \mathbb{R}_{<0}^N
\end{equation}
in $\mathbb{H}^2$, respectively, where $u_*^{\mathbb{E}}$ and $u_*^{\mathbb{H}}$ are points arbitrarily chosen in their respective domains.
\end{definition}
\begin{remark}
These integrals were first introduced by de Verdiere \cite{DV}, who also first proved the
convexity of the functional $F$ in the zero weight case. Chow-Luo generalized de Verdiere's results to any weight case
(see Proposition 3.9 in \cite{2002Combinatorial}). By the branched $\alpha$-potentials \eqref{main ricci potential E} \eqref{main ricci potential H}, we can show the following important proposition.
\end{remark}
\begin{proposition}\label{prop origin potential}
The branched $\alpha$-flows \eqref{main flow E} and \eqref{main flow H} are precisely negative gradient flows of the branched $\alpha$-potentials \eqref{main ricci potential E} and \eqref{main ricci potential H} respectively.
\end{proposition}
\begin{proof}
First we can show that $\sum_{i = 1}^{N}((K_i+2\pi\beta_i)-s_{\alpha}^{\mathbb{E}}r_{i}^{\alpha})\mathrm{d}u^{\mathbb{E}}_{i}$ and $\sum_{i = 1}^{N}((K_i+2\pi\beta_i)-s_{\alpha}^{\mathbb{H}}\tanh^{\alpha}\frac{r_i}{2})\mathrm{d}u_{i}^{\mathbb{H}}$ are both closed forms by Lemma \ref{Lemma 4} and Lemma \ref{Lemma 4.5}. For this reason and both integrals \eqref{main ricci potential E} \eqref{main ricci potential H} lie within a simply connected region of their respective domains, thus \eqref{main ricci potential E} \eqref{main ricci potential H} are well defined. Noting that $\frac{\partial F^{\mathbb{E}}}{\partial u_i^{\mathbb{E}}}=K_i+2\pi\beta_i-s_{\alpha}^{\mathbb{E}}r_{i}^{\alpha}$ in $\mathbb{E}^2$ and $\frac{\partial F^{\mathbb{H}}}{\partial u_i^{\mathbb{H}}}=K_i+2\pi\beta_i-s_{\alpha}^{\mathbb{H}}\tanh^{\alpha}\frac{r_i}{2}$ in $\mathbb{H}^2$, it follows that $F^{\mathbb{E}}$ and $F^{\mathbb{H}}$ define negative gradient flows in $\mathbb{E}^2$ and $\mathbb{H}^2$, respectively, in other words, $\dot{u}^{\mathbb{E}}=-\nabla_{u^{\mathbb{E}}} F^{\mathbb{E}}$ and $\dot{u}^{\mathbb{H}}=-\nabla_{u^{\mathbb{H}}} F^{\mathbb{H}}$.
\end{proof}
\begin{remark}\label{reason not sinh}
If we define the branched $\alpha$-flow in $\mathbb{H}^2$ to be
\begin{equation}\label{wrong flow}
\frac{\mathrm{d}u_i^{\mathbb{H}}}{\mathrm{d}t} = s_{\alpha}^{\mathbb{H}}\sinh^{\alpha}\frac{r_i}{2}-(K_i+2\pi\beta_i),
\end{equation}
and define the corresponding branched $\alpha$-potential in $\mathbb{H}^2$ to be
\begin{equation}\label{wrong potential}
    F^{\mathbb{H}}(u^{\mathbb{H}})=\int_{u_*^{\mathbb{H}}}^{u^{\mathbb{H}}}\sum_{i = 1}^{N}((K_i+2\pi\beta_i)-s_{\alpha}^{\mathbb{H}}\sinh^\alpha\frac{r_i}{2})\mathrm{d}u_{i}^{\mathbb{H}},u^{\mathbb{H}}\in \mathbb{R}_{<0}^N.
\end{equation}
Then
\begin{equation}
\begin{aligned}
\frac{\partial^2 F^\mathbb{H}}{\partial u_i^\mathbb{H}\partial u_j^\mathbb{H}}&=\frac{\partial^2 K_i}{\partial u_i^\mathbb{H}\partial u_j^\mathbb{H}}-\alpha s_{\alpha}^{\mathbb{H}}(\delta_{ij}\sinh^\alpha\frac{r_i}{2}\cosh^2\frac{r_i}{2}-\frac{\sinh^\alpha\frac{r_i}{2}\sinh^\alpha\frac{r_j}{2}\cosh^2\frac{r_j}{2}}{\lVert \sinh\frac{r_i}{2}\rVert_{\alpha}^\alpha})\\
&\not=\frac{\partial^2 K_j}{\partial u_j^\mathbb{H}\partial u_i^\mathbb{H}}-\alpha s_{\alpha}^{\mathbb{H}}(\delta_{ij}\sinh^\alpha\frac{r_j}{2}\cosh^2\frac{r_j}{2}-\frac{\sinh^\alpha\frac{r_j}{2}\sinh^\alpha\frac{r_i}{2}\cosh^2\frac{r_i}{2}}{\lVert \sinh\frac{r_j}{2}\rVert_{\alpha}^\alpha})\\
&=\frac{\partial^2 F^\mathbb{H}}{\partial u_j^\mathbb{H}\partial u_i^\mathbb{H}}
\end{aligned}
\end{equation}
when $i\not =j$ and $\alpha \neq 0$. Thus, $\sum_{i = 1}^{N}((K_i+2\pi\beta_i)-s_{\alpha}^{\mathbb{H}}\sinh^{\alpha}\frac{r_i}{2})\mathrm{d}u_{i}^{\mathbb{H}}$ is not a closed form. Once this condition is not satisfied, the integral will be path-dependent, which implies that both integrals \eqref{main ricci potential E} \eqref{main ricci potential H} are not well-defined. This explains that why we choose \eqref{main flow H} instead of \eqref{wrong flow} as the definition of the branched $\alpha$-flow in $\mathbb{H}^2$.
\end{remark}

For the subsequent discussions, we need some lemmas.
\begin{lemma}\label{Lemma 1}
     Given a function $F\in C^2(\mathbb{R}^N)$. Assume $\text{Hess}(F)\geq0$, $\text{rank}(\text{Hess}(F)) = N - 1$, and $F(u + t(1,\ldots,1)^T)=F(u)$ for any $u\in\mathbb{R}^N$ and $t\in\mathbb{R}$. Denote $U=\{u\in\mathbb{R}^N|\sum_i u_i = 0\}$. Then the Hessian of $F|_{U}$ (considered as a function of $N - 1$ variables) is positive definite.
\end{lemma}
\begin{proof}
    Since $F$ is invariant along the direction $t(1,\ldots,1)^T$, the Hessian matrix $\text{Hess}\, F$ has the only zero eigenvalue with corresponding eigenvector $(1,\ldots,1)^T$. As $t(1,\ldots,1)^T$ is orthogonal to $U$, each eigenvalue of $\left. \text{Hess}\, F \right|_{U}$ is positive. Recalling that the Hessian of a function restricted to a linear subspace (viewed as a quadratic form) is the restriction of the original Hessian to that subspace, it follows that $\left. \text{Hess}\, F \right|_{U}$ is positive definite.
\end{proof}

Denote
\begin{equation}
J^{\mathbb{E}} = \frac{\partial(K_1+2\pi\beta_1,\ldots,K_N+2\pi\beta_N)}{\partial(u_1^{\mathbb{E}},\ldots,u_N^{\mathbb{E}})}=\frac{\partial(K_1,\ldots,K_N)}{\partial(u_1^{\mathbb{E}},\ldots,u_N^{\mathbb{E}})}
\end{equation} and
\begin{equation}
J^{\mathbb{H}} = \frac{\partial(K_1+2\pi\beta_1,\ldots,K_N+2\pi\beta_N)}{\partial(u_1^{\mathbb{H}},\ldots,u_N^{\mathbb{H}})}=\frac{\partial(K_1,\ldots,K_N)}{\partial(u_1^{\mathbb{H}},\ldots,u_N^{\mathbb{H}})}
\end{equation}
as the Jacobian matrices of the curvature maps in $\mathbb{E}^2$ and $\mathbb{H}^2$ respectively.

Let $\Sigma = \operatorname{diag}\{r_1, \ldots, r_N\}$ and $\Sigma^\alpha$ denote its $\alpha$-th power. Similarly, define $$\Pi = \operatorname{diag}\{\tanh\frac{r_1}{2}, \ldots, \tanh \frac{r_N}{2}\}$$ as another diagonal matrix and $\Pi^\alpha$ denotes its $\alpha$-th power. Regarding the Jacobian matrices of the curvature maps $J^{\mathbb{E}}$ and $J^{\mathbb{H}}$, Chow and Luo proved the following two lemmas.

\begin{lemma}[\cite{2002Combinatorial}, Proposition 3.9]\label{Lemma 4}
Given $(M,\mathcal{T},\Phi,\mathbb{E}^2)$ with a circle packing metric $r$ and $K=(K_1,\ldots,K_N)^T$ a discrete Gaussian curvature. $u^{\mathbb{E}}$ is the coordinate transformation of $r$ with $u_i^{\mathbb{E}}=\ln r_i$. Then $J^{\mathbb{E}} = \frac{\partial(K_1,\ldots,K_N)}{\partial(u_1^{\mathbb{E}},\ldots,u_N^{\mathbb{E}})}$ is symmetric and positive semi-definite with rank $N - 1$ and kernel $\{t(1,\ldots,1)^T|t\in\mathbb{R}\}$. Moreover, \[\begin{aligned}J_{ij}^{\mathbb{E}}=\begin{cases}>0,&v_j = v_i,\\<0,&v_j\sim v_i,\\0,&v_j\nsim v_i,v_j\neq v_i.\end{cases}\end{aligned}\]
\end{lemma}

\begin{lemma}[\cite{2002Combinatorial}, \cite{DV}]\label{Lemma 4.5}
Given $(M,\mathcal{T}, \Phi, \mathbb{H}^2)$ with a circle packing metric $r$ and $K$ a discrete Gaussian curvature. $u^{\mathbb{H}}$ is the coordinate transformation of $r$ with $u_i^{\mathbb{H}}=\ln\tanh\frac{r_i}{2}$. Then
$J^{\mathbb{H}} = \frac{\partial(K_1,\ldots,K_N)}{\partial(u_1^{\mathbb{H}},\ldots,u_N^{\mathbb{H}})}=J_A^{\mathbb{H}} + J_B^{\mathbb{H}}$
is symmetric and positive definite, where $J_A^{\mathbb{H}}$ is a positive definite diagonal matrix and $J_B^{\mathbb{H}}$ is symmetric and positive semi-definite with rank $N - 1$ and kernel $\{t(1,\ldots,1)^T|t\in\mathbb{R}\}$. Moreover, $J_B^{\mathbb{H}}$ could be written as

\[\begin{aligned}
    (J_B^{\mathbb{H}})_{ij}=\begin{cases}\sum_{v_k\sim v_i}B_{ik},&v_j = v_i,\\-B_{ij},&v_j\sim v_i,\\0,&v_j\nsim v_i,v_j\neq v_i,\end{cases}\end{aligned}\]
with $B_{ij}>0$.
\end{lemma}

Next, we will prove the strict convexity of the branched $\alpha$-potentials \eqref{main ricci potential E} \eqref{main ricci potential H}.
\begin{theorem}\label{thm convex E}
The branched $\alpha$-potential \eqref{main ricci potential E} in $\mathbb{E}^2$ is strictly convex in the hyperplane $U=\{u\in\mathbb{R}^N|\sum_i u_i = 0\}\subset \mathbb{R}^N$.
\end{theorem}
\begin{proof}
    Set $\Lambda_\alpha^{\mathbb{E}}=\Sigma^{-\frac{\alpha}{2}}J^{\mathbb{E}}\Sigma^{-\frac{\alpha}{2}}$. $\{\lambda_i(\Lambda_\alpha^{\mathbb{E}})\}_{1\le i\le N-1}$ denote its all positive eigenvalues. By direct calculation, we have
\begin{equation}
    \text{Hess}_{u^{\mathbb{E}}} \ F^{\mathbb{E}} = J^{\mathbb{E}} - \alpha s_{\alpha}^{\mathbb{E}} \left( \Sigma^{\alpha} - \frac{r^{\alpha}(r^{\alpha})^T}{\lVert r \rVert_{\alpha}^{\alpha}} \right) = \Sigma^{\frac{\alpha}{2}} \left( \Lambda_{\alpha}^{\mathbb{E}} - \alpha s_{\alpha}^{\mathbb{E}} \left( I - \frac{r^{\frac{\alpha}{2}}(r^{\frac{\alpha}{2}})^T}{\lVert r \rVert_{\alpha}^{\alpha}} \right) \right) \Sigma^{\frac{\alpha}{2}}
\end{equation}in $\mathbb{E}^2$.
By Lemma \ref{Lemma 4}, since $J^{\mathbb{E}}$ is a symmetric matrix whose rank is $N-1$, and $\Sigma$ is a diagonal matrix which is full rank, then we can choose an orthonormal matrix $P=(h_0,h_1,\cdots,h_{N-1})$ such that $P^T \Lambda_\alpha^{\mathbb{E}} P=\operatorname{diag}\{0,\lambda_1(\Lambda_\alpha^{\mathbb{E}}),\cdots,\lambda_{N-1}(\Lambda_\alpha^{\mathbb{E}})\}$, where $h_{i-1}$ denotes the $i$-th column of $P$, then $\Lambda_\alpha^{\mathbb{E}} h_0=0, \Lambda_\alpha^{\mathbb{E}} h_i=\lambda_ih_i,1\le i\le N-1$. By Lemma \ref{Lemma 4}, since the kernel of $J^{\mathbb{E}}$ is $t(1,1,\cdots,1)^T$, then $h_0=\frac{r^{\frac{\alpha}{2}}}{\lVert r^{ \frac{\alpha}{2}}\rVert}$ and $r^{\frac{\alpha}{2}}\perp h_i,1\le i\le N-1.$ Hence
\[
\left(I - \frac{r^{\frac{\alpha}{2}} (r^{\frac{\alpha}{2}})^T}{\|r\|_{\alpha}^{\alpha}}\right)h_0 = 0
\]
and
\[
\left(I - \frac{r^{\frac{\alpha}{2}} (r^{\frac{\alpha}{2}})^T}{\|r\|_{\alpha}^{\alpha}}\right)h_i = h_i, 1 \leq i \leq N - 1
\]
in $\mathbb{E}^2$, which implies
\[
\text{Hess}_{u^{\mathbb{E}}} F^{\mathbb{E}} = \Sigma^{\frac{\alpha}{2}} P \operatorname{diag}\left\{0,\lambda_1(\Lambda_{\alpha}^{\mathbb{E}}) - \alpha s_{\alpha}^{\mathbb{E}}, \ldots, \lambda_{N - 1}(\Lambda_{\alpha}^{\mathbb{E}}) - \alpha s_{\alpha}^{\mathbb{E}}\right\} P^T \Sigma^{\frac{\alpha}{2}}
\]in $\mathbb{E}^2$. Since $\alpha\ge0$ and $\chi(M)\le0$, $\alpha s_{\alpha}^{\mathbb{E}} = \alpha \frac{2\pi \chi(M)}{\|r\|_{\alpha}^{\alpha}}\le 0$. Combined with the fact $\lambda_i^{\mathbb{E}}(\Lambda_{\alpha})>0,1\le i\le N-1$, $\text{Hess}_{u^{\mathbb{E}}}F^{\mathbb{E}} \geq 0$, $\text{rank}(\text{Hess}_{u^{\mathbb{E}}} F^{\mathbb{E}}) = N - 1$.  By Lemma \ref{Lemma 1}, it follows that $F^{\mathbb{E}}$ is strictly convex in $U$.
\end{proof}

\begin{theorem}\label{thm convex H}
The branched $\alpha$-potential \eqref{main ricci potential H} in $\mathbb{H}^2$ is strictly convex in $\mathbb{R}_{<0}^N$.
\end{theorem}
\begin{proof}
By Lemma \ref{Lemma 4.5}, Let $J^{\mathbb{H}}=J_A^{\mathbb{H}}+J_B^{\mathbb{H}}$, where $J_A^{\mathbb{H}}$ is a positive definite diagonal matrix and $J_B^{\mathbb{H}}$ is a positive semi-definite matrix with rank $N - 1$ and kernel $t(1,\ldots,1)^T$. Set $\Lambda_{\alpha,A}^{\mathbb{H}}=\Pi^{-\frac{\alpha}{2}}J_A^{\mathbb{H}}\Pi^{-\frac{\alpha}{2}}$,
$\Lambda_{\alpha,B}^{\mathbb{H}}=\Pi^{-\frac{\alpha}{2}}J_B^{\mathbb{H}}\Pi^{-\frac{\alpha}{2}}$. By direct calculation, we have
\begin{equation}
\begin{aligned}
    \text{Hess}_u^{\mathbb{H}} \ F^{\mathbb{H}} &= J^{\mathbb{H}} - \alpha s_{\alpha}^{\mathbb{H}} \left( \Pi^\alpha-\frac{\tanh^{\alpha}\frac{r}{2}(\tanh^\alpha\frac{r}{2})^T}{\lVert \tanh \frac{r}{2} \rVert_{\alpha}^{\alpha}} \right) \\&= \Pi^{\frac{\alpha}{2}} \left( \Lambda_{\alpha,A}^{\mathbb{H}}+ \Lambda_{\alpha,B}^{\mathbb{H}} - \alpha s_{\alpha}^{\mathbb{H}} \left( I - \frac{\tanh^{\frac{\alpha}{2}}\frac{r}{2}(\tanh^{\frac{\alpha}{2}}\frac{r}{2})^T}{\lVert \tanh \frac{r}{2} \rVert_{\alpha}^{\alpha}} \right) \right)  \Pi^{\frac{\alpha}{2}}
\end{aligned}
\end{equation}
in $\mathbb{H}^2$. Since $J_B^{\mathbb{H}}$ is a symmetric matrix whose rank is $N-1$ and $\Pi$ is a diagonal matrix that is full rank, we can choose an orthonormal matrix $Q=(g_0,g_1,\cdots,g_{N-1})$ such that
$Q^T \Lambda_{\alpha,B}^{\mathbb{H}} Q=\operatorname{diag}\{0,\lambda_1(\Lambda_{\alpha,B}^{\mathbb{H}}),\cdots,\lambda_{N-1}(\Lambda_{\alpha,B}^{\mathbb{H}})\}$, where $g_{i-1}$ denotes the $i$-th column of $Q$, then $\Lambda_{\alpha,B}^{\mathbb{H}} g_0=0,\Lambda_{\alpha,B}^{\mathbb{H}} g_i=\lambda_ig_i,1\le i\le N-1$. By Lemma \ref{Lemma 4.5}, since the kernel of $J_B$ is $t(1,1,\cdots,1)^T$, then $g_0=\frac{\tanh^\frac{\alpha}{2}\frac{r}{2}}{\lVert \tanh^\frac{\alpha}{2}\frac{r}{2}\rVert}$ and $\tanh^\frac{\alpha}{2}\frac{r}{2}\perp g_i,1\le i\le N-1.$ Hence
\[
\left(I - \frac{\tanh^{\frac{\alpha}{2}}\frac{r}{2}(\tanh^{\frac{\alpha}{2}}\frac{r}{2})^T}{\lVert \tanh \frac{r}{2} \rVert_{\alpha}^{\alpha}} \right)g_0 = 0
\]
and
\[
\left(I - \frac{\tanh^{\frac{\alpha}{2}}\frac{r}{2}(\tanh^{\frac{\alpha}{2}}\frac{r}{2})^T}{\lVert \tanh \frac{r}{2} \rVert_{\alpha}^{\alpha}} \right)g_i = g_i, 1 \leq i \leq N - 1
\]
in $\mathbb{H}^2$. Set
\[
    L=Q^T\left( \Lambda_{\alpha,B}^{\mathbb{H}} -  \alpha s_{\alpha}^{\mathbb{H}} \left( I - \frac{\tanh^{\frac{\alpha}{2}}\frac{r}{2}(\tanh^{\frac{\alpha}{2}}\frac{r}{2})^T}{\lVert \tanh \frac{r}{2} \rVert_{\alpha}^{\alpha}}  \right) \right)Q,
\]
then the above reasoning implies
\[\begin{aligned}
    L=\begin{pmatrix}
0 & \\
 & \lambda_1(\Lambda_{\alpha,B}^{\mathbb{H}}) - \alpha s_{\alpha}^{\mathbb{H}}& &\\
& & \lambda_2(\Lambda_{\alpha,B}^{\mathbb{H}}) - \alpha s_{\alpha}^{\mathbb{H}} & &\\
& & &  \ddots &  \\
& & & & &  \lambda_{N-1}(\Lambda_{\alpha,B}^{\mathbb{H}}) - \alpha s_{\alpha}^{\mathbb{H}}\end{pmatrix}.
\end{aligned}
\]Since $\alpha\ge0$ and $\chi(M)\le0$, $\alpha s_\alpha^{\mathbb{H}}=\alpha \frac{2\pi\chi(M)}{\|\tanh \frac{r}{2}\|_{\alpha}^{\alpha}}\leq 0$. Then $L$ is positive semi-definite with rank $N-1$. Since the positive definiteness is invariant under congruent transformation, $P^T\Lambda_{\alpha,A}^{\mathbb{H}}P$ is positive definite. Note that the sum of a positive definite matrix and a positive semi-definite matrix is still positive definite, thereby
\[
\text{Hess}_{u^{\mathbb{H}}} F^{\mathbb{H}} = \Pi^{\frac{\alpha}{2}} Q (Q^T\Lambda_{\alpha,A}^{\mathbb{H}}Q+L) Q^T \Pi^{\frac{\alpha}{2}}>0,
\]
 it follows that $F^{\mathbb{H}}$ is strictly convex in $\mathbb{R}_{<0}^N$.
\end{proof}

\section{Exponential Convergence}\label{section 5}
\hspace{14pt}
In this section, after showing the properness of the branched $\alpha$-potentials \eqref{main ricci potential E} and \eqref{main ricci potential H}, we prove the convergence of the solutions to the branched $\alpha$-flows \eqref{main flow E} \eqref{main flow H}. \par
First, we demonstrate the existence of $u_H^{\mathbb{E}}$ and $u_H^{\mathbb{H}}$.
\begin{lemma}[\cite{2002Combinatorial}]\label{injective}
     Suppose $\Omega\subset\mathbb{R}^N$ is convex, the function $h:\Omega\to\mathbb{R}$ is strictly convex, then the map $\nabla h:\Omega\to\mathbb{R}^N$ is injective.
\end{lemma}

\begin{theorem}\label{thm critical point E}
The branched $\alpha$-potential $F^{\mathbb{E}}$ \eqref{main ricci potential E} satisfies $\lim_{\lVert u^{\mathbb{E}}\rVert\to +\infty} F^{\mathbb{E}}(u^{\mathbb{E}})=+\infty$ in $U=\{u^{\mathbb{E}}\in\mathbb{R}^N|\sum_i u_i^{\mathbb{E}} = 0\}\subset \mathbb{R}^N$, therefore $F^{\mathbb{E}}|_U$ is proper and has a unique critical point. Furthermore, there exists a unique constant branched $\alpha$-metric ${r_*}_H^{\mathbb{E}}$ in U.
\end{theorem}
\begin{proof}
Recall that
\[
F^{\mathbb{E}}(u^{\mathbb{E}})=\int_{u_*^{\mathbb{E}}}^{u^{\mathbb{E}}}\sum_{i = 1}^{N}((K_i+2\pi\beta_i)-s_{\alpha}^{\mathbb{E}}r_{i}^{\alpha})\mathrm{d}u^{\mathbb{E}}_{i},u^{\mathbb{E}}\in \mathbb{R}^N.
\]Denote $r_*=({r_*}_1,\ldots,{r_*}_N)=(\exp({u_*^{\mathbb{E}}}_1),\ldots,\exp({u_*^{\mathbb{E}}}_N))=\exp(u_*^{\mathbb{E}})$. Scaling \(r_*\) to any \(cr_*\) with \(c > 0\), the corresponding branched $\alpha$-curvature \eqref{curvature E} is multiplied by a constant $1/c^{\alpha}$, and the corresponding branched $\alpha$-flow \eqref{main flow E} is still invariant. We term this property the scaling equivalence. Hence, in the subsequent content of this section, we may suppose \(r_*\) lies in the hypersurface \(\prod_{1\le i\le N} {r_*}_i = 1\) , that is, $\sum_{1\leq i\leq N} {u_*^{\mathbb{E}}}_i=0$. In the following, we consider $u^{\mathbb{E}}\in U$.

It follows from the definition of $F^{\mathbb{E}}$ and the fact (\ref{main ricci potential E}) is a closed $1$-form in $\mathbb{R}^N$ that,
\begin{equation}\label{sum potentials E}
\begin{aligned}
F^{\mathbb{E}}(u^{\mathbb{E}}) &= \int_{u_{*_1}^{\mathbb{E}}}^{u_1^{\mathbb{E}}} [K_1(s_1, u_{*_2}^{\mathbb{E}}, u_{*_3}^{\mathbb{E}}, \ldots,u_{*_N}^{\mathbb{E}}) + 2\beta_1\pi-s_\alpha^{\mathbb{E}}r_1^\alpha] \mathrm{d}s_1 + \cdots \\
& + \int_{u_{*_k}^{\mathbb{E}}}^{u_k^{\mathbb{E}}} [K_k(u_{1}^{\mathbb{E}}, \ldots, u_{{k-1}}^{\mathbb{E}}, s_k, u_{*_{k+1}}^{\mathbb{E}}, \ldots,u_{*_N}^{\mathbb{E}}) + 2\beta_k\pi-s_\alpha^{\mathbb{E}}r_k^\alpha] \mathrm{d}s_k + \cdots \\
& + \int_{u_{*_l}^{\mathbb{E}}}^{u_l^{\mathbb{E}}} [K_l(u_{1}^{\mathbb{E}}, \ldots, u_{l-1}^{\mathbb{E}}, s_l, u_{*_{l+1}}^{\mathbb{E}}, \ldots,  u_{*_{N}}^{\mathbb{E}}) + 2\beta_l\pi-s_\alpha^{\mathbb{E}}r_l^\alpha] \mathrm{d}s_l + \cdots \\
& + \int_{u_{*_N}^{\mathbb{E}}}^{u_N^{\mathbb{E}}} [K_N(u_{1}^{\mathbb{E}},u_{2}^{\mathbb{E}}, u_{3}^{\mathbb{E}}, \ldots, u_{N-1}^{\mathbb{E}}, s_N) + 2\beta_N\pi-s_\alpha^{\mathbb{E}}r_N^\alpha] \mathrm{d}s_N.
\end{aligned}
\end{equation}

Note that $u^ {\mathbb{E}}\in U$, that is, $\sum_{i = 1}^{N}u_i^{\mathbb{E}} = 0$. When $\lVert u^{\mathbb{E}}\rVert \to\infty$, it implies that there exist at least two components $u_k^{\mathbb{E}}$ and $u_l^{\mathbb{E}}$ such that $u_k^{\mathbb{E}} \to+\infty$ and $u_l^{\mathbb{E}}\to-\infty$.

\begin{enumerate}[(i)]
    \item $u_k^{\mathbb{E}} \to+\infty$

Without loss of generality, assume that the number of $u_j^{\mathbb{E}(v_k)}$ (where $u_j^{\mathbb{E}(v_k)}$ is associated with the vertices adjacent to $v_k^{\mathbb{E}}$) that tend to $\infty$ is no more than 2. Since $F^{\mathbb{E}}(u^{\mathbb{E}})=F^{\mathbb{E}}(u^{\mathbb{E}}-2u_k^{\mathbb{E}})$, if not, considering $F^{\mathbb{E}}(u^{\mathbb{E}}-2u_k^{\mathbb{E}})$ is equivalent to the case of $u_k^{\mathbb{E}} \to-\infty$, which is (\ref{case 2}). By $u_i^{\mathbb{E}}=\ln r_i$ and Lemma \ref{Lemma eps}, we get that there exists a positive constant $L_k$ such that for any $u_k^{\mathbb{E}} > L_k$, $K_k(u^{\mathbb{E}})=2\pi-\theta_k\geq\gamma_k > 0 \geq s_\alpha^{\mathbb{E}}r_k^\alpha \geq 2\pi \chi(M)$. Thus we obtain

\begin{equation}\label{case 1 int}
\begin{aligned}
&\int_{u_{*_k}^{\mathbb{E}}}^{u_k^{\mathbb{E}}} [K_k(u_{1}^{\mathbb{E}}, \ldots, u_{{k-1}}^{\mathbb{E}}, s_k, u_{*_{k+1}}^{\mathbb{E}}, \ldots,u_{*_N}^{\mathbb{E}}) + 2\beta_k\pi-s_\alpha^{\mathbb{E}}r_k^\alpha] \mathrm{d}s_k > (\gamma_k + 2\beta_k\pi)(u_k^{\mathbb{E}} - L_k) \\
&+ \int_{u_{*_k}^{\mathbb{E}}}^{L_k} [K_k(u_{1}^{\mathbb{E}}, \ldots, u_{{k-1}}^{\mathbb{E}}, s_k, u_{*_{k+1}}^{\mathbb{E}}, \ldots,u_{*_N}^{\mathbb{E}}) + 2\beta_k\pi-s_\alpha^{\mathbb{E}}r_k^\alpha] \mathrm{d}s_k.
\end{aligned}
\end{equation}

\item\label{case 2} $u_l^{\mathbb{E}} \to-\infty$

As $u_l^{\mathbb{E}} \to -\infty$ implies $r_l \to 0$, by Lemma \ref{Lemma eps}, we have $\theta_l = \sum_{v_l} \theta_l^{(v_l)} \to \sum_{v_l} \left[ \pi - \varPhi(e^{(v_l)}) \right]$ as $u_l^{\mathbb{E}} \to -\infty$. Given that $br(H)$ is a branch structure for the weighted triangulation $(M,\mathcal{T},\Phi)$, it follows that $\sum_{v_l} \left[ \pi - \varPhi(e^{(v_l)}) \right] - (2\beta_l + 2)\pi > 0$ by definition \ref{def branch}. Thus, there exists a constant $L_l < 0$ such that for all $u_l^{\mathbb{E}} < L_l$, $\theta_l - (2\beta_l + 2)\pi \geq \gamma_l$, where $\gamma_l > 0$ is a constant. Notice when $r_l\to 0$, there is $s_\alpha^{\mathbb{E}}r_l^\alpha\to 0^-$. Therefore, we obtain

\begin{equation}\label{case 2 int}
\begin{aligned}
&\int_{u_{*_l}^{\mathbb{E}}}^{u_l^{\mathbb{E}}} [K_l(u_{1}^{\mathbb{E}}, \ldots, u_{l-1}^{\mathbb{E}}, s_l, u_{*_{l+1}}^{\mathbb{E}}, \ldots,  u_{*_{N}}^{\mathbb{E}}) + 2\beta_l\pi-s_\alpha^{\mathbb{E}}r_l^\alpha] \mathrm{d}s_l \\
&\geq \int_{u_l^{\mathbb{E}}}^{\ L_l} \gamma_l \mathrm{d}s_l + \int_{u_{*_l}^{\mathbb{E}}}^{\ L_l} [(2\beta_l + 2)\pi - \theta_l] \mathrm{d}s_l \\
&= \gamma_l(L_l - u_l^{\mathbb{E}}) + \int_{u_{*_l}^{\mathbb{E}}}^{\ L_l} [(2\beta_l + 2)\pi - \theta_l] \mathrm{d}s_l.
\end{aligned}
\end{equation}

Thus, when $\|u^{\mathbb{E}}\| \to +\infty$, both equation (\ref{case 1 int}) and equation (\ref{case 2 int}) tend to $+\infty$. By applying the same reasoning as presented earlier, it is straightforward to observe that, as $\|u^{\mathbb{E}}\| \to \infty$, the remaining terms on the right-hand side of equation \eqref{sum potentials E} are either finite or approach $+\infty$. Consequently, we conclude that within the subspace $U$, if $\|u^{\mathbb{E}}\| \to \infty$, then $F^{\mathbb{E}}(u^{\mathbb{E}})\to+\infty$.
\end{enumerate}

Since $\lim_{\|u^{\mathbb{E}}\|\to\infty}F^{\mathbb{E}}(u^{\mathbb{E}})=+\infty$, there exists a positive number $X$ such that when $\lVert u^{\mathbb{E}} \rVert>X$ then $F^{\mathbb{E}}(u^{\mathbb{E}})>F^{\mathbb{E}}(0)+1$. For the closed ball $B(0,X)\subset\mathbb{R}^N$, when $u^{\mathbb{E}}\not\in B(0,X)\cap U$, $F^{\mathbb{E}}(u^{\mathbb{E}})>F^{\mathbb{E}}(0)+1$. Notice $B(0,X)\cap U$ is compact in $U$ and $F^{\mathbb{E}}$ is continuous, so there exists a point $u^{\mathbb{E}}_c\in B(0,X)\cap U$ such that $F(u^{\mathbb{E}}_c)=\min_{u^{\mathbb{E}}\in B(0,X)\cap U}F^{\mathbb{E}}(u^{\mathbb{E}})$. This leads to $F^{\mathbb{E}}(u^{\mathbb{E}})>F^{\mathbb{E}}(0)+1> F^{\mathbb{E}}(u^{\mathbb{E}}_c),u^{\mathbb{E}}\not\in B(0,X)\cap U$ and $F^{\mathbb{E}}(u^{\mathbb{E}})\geq F^{\mathbb{E}}(u^{\mathbb{E}}_c),u^{\mathbb{E}}\in B(0,X)\cap U$, which means $u^{\mathbb{E}}_c$ is a
critical point. The uniqueness of the critical point is ensured by the strict convexity of $F^{\mathbb{E}}$ in $U$ by Theorem \ref{thm convex E} and Lemma \ref{injective}. Moreover, it is easy to see that $u^{\mathbb{E}}_c$ is the minimal point of $F^{\mathbb{E}}|_U$ and therefore $F^{\mathbb{E}}|_U$ is proper. Since $\frac{\partial F^{\mathbb{E}}}{\partial u_i^{\mathbb{E}}}=K_i+2\pi\beta_i-s_{\alpha}^{\mathbb{E}}r_{i}^{\alpha}$ in $\mathbb{E}^2$, a critical point of $F^{\mathbb{E}}$ is precisely a constant branched $\alpha$-metric in $U$. This completes the proof.
\end{proof}

\begin{remark}\label{r_H}
The above theorem guarantees the existence and uniqueness of a constant branched $\alpha$-curvature circle packing metric \({r_*}_H^{\mathbb{E}}\) in $U$ relying only on topological-combinatorial conditions. Furthermore, for any positive constant $c$, \(c{r_*}_H^{\mathbb{E}}\) is also a constant branched $\alpha$-metric in $\mathbb{E}^2$. If $r$ is a constant branched $\alpha$-metric in $\mathbb{E}^2$, then \(r/\prod_i r_i\) is equivalent to \({r_*}_H^{\mathbb{E}}\). Consequently, all constant branched $\alpha$-metrics in $\mathbb{E}^2$ form a ray \(\{c{r_*}_H^{\mathbb{E}}\}_{c > 0}\).
\end{remark}
\begin{theorem}\label{thm critical point H}
The branched $\alpha$-potential $F^{\mathbb{H}}$ satisfies $\lim_{\lVert u^{\mathbb{H}}\rVert\to \infty} F^{\mathbb{H}}(u^{\mathbb{H}})=+\infty$ in $\mathbb{R}_{<0}^N$. Additionally, $F^{\mathbb{H}}$ is proper and has a unique critical point in $\mathbb{R}_{<0}^N$. Furthermore, there exists a unique constant branched $\alpha$-metric $r_H^{\mathbb{H}}$ in $\mathbb{R}_{<0}^N$.
\end{theorem}
\begin{proof}
First we have
\begin{equation}
\begin{aligned}
F^{\mathbb{H}}(u^{\mathbb{H}}) &= \int_{u_{*_1}^{\mathbb{H}}}^{u_1^{\mathbb{H}}} [K_1(s_1, u_{*_2}^{\mathbb{H}}, u_{*_3}^{\mathbb{H}}, \ldots,u_{*_N}^{\mathbb{H}}) + 2\beta_1\pi-s_{\alpha}^{\mathbb{H}}\tanh^{\alpha}\frac{r_1}{2}] \mathrm{d}s_1 + \cdots \\
&+ \int_{u_{*_k}^{\mathbb{H}}}^{u_k^{\mathbb{H}}} [K_k(u_{1}^{\mathbb{H}}, \ldots, u_{{k-1}}^{\mathbb{H}}, s_k, u_{*_{k+1}}^{\mathbb{H}}, \ldots,u_{*_N}^{\mathbb{H}}) + 2\beta_k\pi-s_{\alpha}^{\mathbb{H}}\tanh^{\alpha}\frac{r_k}{2}] \mathrm{d}s_k + \cdots \\
&+ \int_{u_{*_l}^{\mathbb{H}}}^{u_l^{\mathbb{H}}} [K_l(u_{1}^{\mathbb{H}}, \ldots, u_{l-1}^{\mathbb{H}}, s_l, u_{*_{l+1}}^{\mathbb{H}}, \ldots,  u_{*_{N}}^{\mathbb{H}}) + 2\beta_l\pi-s_{\alpha}^{\mathbb{H}}\tanh^{\alpha}\frac{r_l}{2}] \mathrm{d}s_l + \cdots \\
&+ \int_{u_{*_N}^{\mathbb{H}}}^{u_N^{\mathbb{H}}} [K_N(u_{1}^{\mathbb{H}},u_{2}^{\mathbb{H}}, u_{3}^{\mathbb{H}}, \ldots, u_{N-1}^{\mathbb{H}}, s_N) + 2\beta_N\pi-s_{\alpha}^{\mathbb{H}}\tanh^{\alpha}\frac{r_N}{2}] \mathrm{d}s_N.
\end{aligned}
\end{equation}
If $\|u^{\mathbb{H}}\|\to +\infty$, note that $u^{\mathbb{H}}\in(-\infty, 0)^N$, so it is sufficient to focus only on the situation where $u_k^{\mathbb{H}}\to-\infty$. Given the relation $u_k^{\mathbb{H}} = \ln\tanh(r_k/2)$, we know that as $u_k^{\mathbb{H}}\to-\infty$, $r_k\to 0$. Using arguments analogous to those used in the Euclidean case, we can show that when $u_k^{\mathbb{H}}\to-\infty$, there exists a constant $\tilde{L}_k < 0$ such that for any $u_k^{\mathbb{H}} < \tilde{L}_k$, the following inequality holds:
$\theta_{k}-(2\beta_k + 2)\pi\geq \tilde{\gamma}_k > 0$. Notice when $r_k\to 0$, there is $s_\alpha^{\mathbb{H}}r_k^\alpha\to 0^-$. Consequently, we have:
\begin{equation}
\begin{aligned}
&\int_{u_{*_k}^{\mathbb{H}}}^{u_k^{\mathbb{E}}} [K_k(u_{1}^{\mathbb{H}}, \ldots, u_{{k-1}}^{\mathbb{H}}, s_k, u_{*_{k+1}}^{\mathbb{H}}, \ldots,u_{*_N}^{\mathbb{H}}) + 2\beta_k\pi-s_{\alpha}^{\mathbb{H}}\tanh^{\alpha}\frac{r_i}{2}] \mathrm{d}s_k \\
&\geq \int_{u_k^{\mathbb{H}}}^{\tilde{L}_k} \tilde{\gamma}_k \mathrm{d}s_k + \int_{u_{*_k}^{\mathbb{H}}}^{\tilde{L}_k} [(2\beta_k + 2)\pi - \theta_k] \mathrm{d}s_k \\
&= \tilde{\gamma}_k(\tilde{L}_k - u_k^{\mathbb{H}}) + \int_{u_{*_k}^{\mathbb{H}}}^{\tilde{L}_k} [(2\beta_k + 2)\pi - \theta_k] \mathrm{d}s_k.
\end{aligned}
\end{equation}
Thus, $F^{\mathbb{H}}(u^{\mathbb{H}})\to +\infty$ as $\|u^{\mathbb{H}}\|\to \infty$.

To show $F^{\mathbb{H}}$ has a critical point, we carry out the following discussion.
It is evident that when considering $u^{\mathbb{H}}\to 0$, we can assume, without loss of generality, that $u_l^{\mathbb{H}}\to 0$ and $u_i^{\mathbb{H}}\neq 0$ for $i\neq l$. Given that $r_l\to +\infty$ as $u_l^{\mathbb{H}}\to 0$, there exists a constant $\tilde{L}_l < 0$ such that for any $u_l^{\mathbb{H}}$ satisfying $\tilde{L}_l < u_l^{\mathbb{H}} < 0$, we have $K_l = 2\pi-\theta_{l}\geq\tilde{\gamma}_l > 0 > s_{\alpha}^{\mathbb{H}}\tanh^{\alpha}\frac{r_l}{2} > 2\pi \chi(M)$. Consequently, whenever $u_l^{\mathbb{H}} > \tilde{L}_l$, we obtain that
\begin{equation}
\begin{aligned}
&\int_{u_{*_l}^{\mathbb{H}}}^{u_l^{\mathbb{H}}} [K_l(u_{1}^{\mathbb{H}}, \ldots, u_{l-1}^{\mathbb{H}}, s_l, u_{*_{l+1}}^{\mathbb{H}}, \ldots,  u_{*_{N}}^{\mathbb{H}}) + 2\beta_l\pi-s_{\alpha}^{\mathbb{H}}\tanh^{\alpha}\frac{r_l}{2}] \mathrm{d}s_l \\
&\geq \int_{u_{*_l}^{\mathbb{H}}}^{\tilde{L}_l} [K_l(u_{1}^{\mathbb{H}}, \ldots, u_{l-1}^{\mathbb{H}}, s_l, u_{*_{l+1}}^{\mathbb{H}}, \ldots,  u_{*_{N}}^{\mathbb{H}}) + 2\beta_l\pi-s_{\alpha}^{\mathbb{H}}\tanh^{\alpha}\frac{r_l}{2}] \mathrm{d}s_l\\ &+ \int_{\tilde{L}_l}^{u_l^{\mathbb{H}}} [\tilde{\gamma}_l + 2\beta_l\pi-s_{\alpha}^{\mathbb{H}}\tanh^{\alpha}\frac{r_l}{2}] \mathrm{d}s_l \\
&> \int_{u_{*_l}^{\mathbb{H}}}^{\tilde{L}_l} [K_l(u_{1}^{\mathbb{H}}, \ldots, u_{l-1}^{\mathbb{H}}, s_l, u_{*_{l+1}}^{\mathbb{H}}, \ldots,  u_{*_{N}}^{\mathbb{H}}) + 2\beta_l\pi-s_{\alpha}^{\mathbb{H}}\tanh^{\alpha}\frac{r_l}{2}] \mathrm{d}s_l + [\tilde{\gamma}_l + 2\beta_l\pi](u_l^{\mathbb{H}}- \tilde{L}_l).
\end{aligned}
\end{equation}

Let $l = N$. Then we can show that there exists a $\delta_N < 0$ such that $f_N(u)=F^{\mathbb{H}}(u_1^{\mathbb{H}},\ldots,u_{N-1}^{\mathbb{H}},u)$ is a strictly increasing function on the interval $(\delta_N, 0)$. By symmetry, for \[f_l(u) = F^{\mathbb{H}}(u_1^{\mathbb{H}},\ldots,u_{l-1}^{\mathbb{H}},u,u_{l+1}^{\mathbb{H}},\ldots,u_N^{\mathbb{H}}),\] there also exists a $\delta_l < 0$ such that $f_l(u)$ is strictly increasing on the interval $(\delta_l, 0)$. Set $\delta=\max_{1\le i\le N}\delta_i$.

Since $F^{\mathbb{H}}\in C^2(\mathbb{R}_{<0}^N)$, so $F^{\mathbb{H}}$ can be uniquely extended to $\tilde{F}^{\mathbb{H}}$ which is defined on $\mathbb{R}_{\leq 0}^N$. Similar to the proof of Theorem \ref{thm critical point E}, it is easy to see that there exists a $Y > 1$ such that if $\tilde{F}^{\mathbb{H}}$ has a critical point, then it must lie within the ball $B(0, Y)\cap \mathbb{R}_{\leq 0}^N$. Since in $B(0,\mid\delta\mid)\cap \mathbb{R}_{< 0}^N$, $f_i$ is strictly increasing for $1\le i\le N$, there must be a critical point within $B(0, Y)\cap \mathbb{R}_{< 0}^N\subset \mathbb{R}_{<0}^N$. From Theorem \ref{thm convex H} and Lemma \ref{injective}, the uniqueness of the critical point is ensured by the strict convexity of $F^{\mathbb{H}}$ in $\mathbb{R}_{<0}^N$. Moreover, it is easy to see $F^{\mathbb{H}}$ is proper. Since $\frac{\partial F^{\mathbb{H}}}{\partial u_i^{\mathbb{H}}}=K_i+2\pi\beta_i-s_{\alpha}^{\mathbb{H}}\tanh^{\alpha}\frac{r_i}{2}$, a critical point of $F^{\mathbb{H}}$ is precisely a constant branched $\alpha$-metric in $\mathbb{H}^2$. Therefore, we complete the proof.
\end{proof}

\begin{theorem}\label{main thm}
For any initial metric $r(0)$, the solution $r(t)$ to \eqref{main flow E} converges exponentially fast to a constant $\alpha$-metric $r_H^{\mathbb{E}}$ in $\mathbb{E}^2$. Similarly, the solution $r(t)$ to \eqref{main flow H} converges exponentially fast to the constant $\alpha$-metric $r_H^
{\mathbb{H}}$ in $\mathbb{H}^2$.
\end{theorem}
\begin{proof}
The existence and uniqueness of constant branched $\alpha$-metrics ${u_*}_H^{\mathbb{E}}$ in $U$ and $u_H^{\mathbb{H}}$ in $\mathbb{R}_{<0}^N$ are first established via Theorems \ref{thm critical point E} and \ref{thm critical point H}. We claim that $u^{\mathbb{E}}(t)$ lies in a compact subset of $U$ and $u^{\mathbb{H}}(t)$ lies in a compact subset of $\mathbb{R}_{<0}^N$. In the case of $\mathbb{H}^2$, define $p^{\mathbb{H}}(t) = F^{\mathbb{H}}(u^{\mathbb{H}}(t))$. By Theorem \ref{thm critical point H}, $F^{\mathbb{H}}$ is proper in $\mathbb{R}_{<0}^N$, which implies that $p^{\mathbb{H}}(t)$ is bounded below. We can derive that
\[
{p^{\mathbb{H}}}^{\prime}(t)=-\sum_{i}(K_{i}+2\pi\beta_i-s_{\alpha}^{\mathbb{H}}\tanh^{\alpha}\frac{r_i}{2})^{2}=-\|\nabla F^{\mathbb{H}}\|^{2}\leq0,
\]
and
\[
{p^{\mathbb{H}}}^{\prime\prime}(t)=2\left(K+\beta - s_{\alpha}^{\mathbb{H}}\tanh^{\alpha}\frac{r_i}{2}\right)^{T}\text{Hess}_{u^{\mathbb{H}}}F^{\mathbb{H}}\left(K+\beta - s_{\alpha}^{\mathbb{H}}\tanh^{\alpha}\frac{r_i}{2}\right)\geq0.
\]
Given that ${p^{\mathbb{H}}}^{\prime}\leq0$, ${p^{\mathbb{H}}}^{\prime\prime}\geq0$, and $p^{\mathbb{H}}$ is bounded from below, it follows that ${p^{\mathbb{H}}}^{\prime}(+\infty)=0$.
If $u^{\mathbb{H}}$ is not bounded in $\mathbb{R}_{<0}^N$, then by Theorem \ref{thm critical point H}, we get $\lim_{t\to +\infty}p^{\mathbb{H}}(t)=+\infty$, which is impossible.

Since $u_H^{\mathbb{H}}$ is the unique critical point of $F^{\mathbb{H}}$ by Theorem \ref{thm critical point H} and $\{u^{\mathbb{H}}(t)\}\subset\subset \mathbb{R}_{<0}^N$, combining with the fact that ${p^{\mathbb{H}}}^{\prime}(+\infty)=0$, we deduce that there exist a sequence $\xi_n\to +\infty$ such that  $u^{\mathbb{H}}(\xi_n)$ converges to $u_H^{\mathbb{H}}$. Let $\Gamma_i^{\mathbb{H}}$ be the right-hand side of the equation \eqref{main flow H}, and $\Gamma^{\mathbb{H}}$ denotes $(\Gamma_1^{\mathbb{H}},\ldots,\Gamma_N^{\mathbb{H}})^T$, then we have $D \Gamma^{\mathbb{H}}|_{u_H^{\mathbb{H}}}=-\text{Hess}_{u_H^{\mathbb{H}}}F^{\mathbb{H}}$. So $D\Gamma^{\mathbb{H}}|_{u_H^{\mathbb{H}}}$ is a Hurwitz matrix by Theorem \ref{thm convex H}, which implies $u_H^{\mathbb{H}}$ is a local attractor of \eqref{main flow H}. By the Lyapunov Stability Theorem, we conclude that $u^{\mathbb{H}}(t)$ converges exponentially fast to $u_H^{\mathbb{H}}$.

In the case of $\mathbb{E}^2$, By virtue of the scaling equivalence stated in the proof of Theorem \ref{thm critical point E} and Remark \ref{r_H}, we can replace \(u^{\mathbb{E}}\) with \(u^{\mathbb{E}}-\sum_{1\leq i\leq N}u_i^{\mathbb{E}}\). Define $p^{\mathbb{E}}(t) = F^{\mathbb{E}}|_U(u^{\mathbb{E}}(t))$, by similar arguments, we can also obtain that $u^{\mathbb{E}}(t)$ converges exponentially fast to $u_H^{\mathbb{E}}$.
\end{proof}

\section{Branched $\alpha$-Flows with Prescribed Curvatures}\label{sec 6}
\hspace{14pt}
In this section, we address the prescribed curvature problems by introducing the corresponding branched $\alpha$-curvatures and branched $\alpha$-flows, which generalize previous results by relaxing the precondition \(\chi(M) \leq 0\). Specifically, our framework applies to surfaces with arbitrary Euler characteristic \(\chi(M) \in \mathbb{Z}\).
\begin{definition}
    Suppose $(M,\mathcal{T},\Phi,P)$ is a branched weighted triangulated closed surface with a circle packing metric $r$, where $P$ is a circle packing with a branch structure \[br(P)=\{(b_1,\beta_1),(b_2,\beta_2),\ldots,(b_m,\beta_m)\}\subset V\times m,\] Given a function defined: $\overline{R}\in C(V)$,  a modified branched $\alpha$-flow with respect to $\overline{R}$ in $\mathbb{E}^2$ is defined to be
\begin{equation}\label{pre flow E}
\frac{\mathrm{d}u_i^{\mathbb{E}}}{\mathrm{d}t}=\overline{R}_ir_i^{\alpha}-(K_i+2\pi \beta_i).
\end{equation}
A modified branched $\alpha$-$\tanh$-flow with respect to $\overline{R}$ in $\mathbb{H}^2$ is defined to be
\begin{equation}\label{pre tanh flow H}
\frac{\mathrm{d}u_i^{\mathbb{H}}}{\mathrm{d}t}=\overline{R}_i\tanh^{\alpha}\frac{r_i}{2}-(K_i+2\pi \beta_i).
\end{equation}

Recalling the branched $\alpha$-curvatures \eqref{curvature E} and \eqref{curvature H}, we rewrite $B_{\alpha,i}^{\mathbb{H}}$ as $B_{\alpha,i,\tanh}^{\mathbb{H}}$ called the branched $\alpha$-$\tanh$-curvature in $\mathbb{H}^2$. $\overline{R}$ is called \emph{admissible} in $\mathbb{E}^2$ if there is a circle packing metric $\overline{r}^{\mathbb{E}}$ with the branched $\alpha$-curvature $\overline{R}$ in $\mathbb{E}^2$; $\overline{R}$ is called \emph{$\tanh$-admissible} in $\mathbb{H}^2$ if there is a circle packing metric $\overline{r}_{\tanh}^{\mathbb{H}}$ with the branched $\tanh$-$\alpha$-curvature $\overline{R}$ in $\mathbb{H}^2$.
\end{definition}

These flows (\ref{pre flow E}) (\ref{pre tanh flow H}) can be used to study the prescribed curvature problems. On one hand, if the solutions to the flows (\ref{pre flow E}) and (\ref{pre tanh flow H}) exist for all time and converge, then it is easy to see that $\overline{R}$ is admissible in $\mathbb{E}^2$ and $\tanh$-admissible in $\mathbb{H}^2$, respectively. On the other hand, we have the following theorem.

\begin{theorem}\label{thm pre}

Suppose $(M,\mathcal{T},\Phi,H)$ is a branched weighted triangulated surface, $\alpha$ is a non-negative real number and $\overline{R}\in C(V)$ is a function defined on $\mathcal{T}$. If $\overline{R}_i\leq0$ for all $i$, and $\{\alpha\overline{R}_i\}_{1\leq i\leq N}$ are not identically zero. Additionally, $\overline{R}$ is admissible and $\tanh$-admissible by metrics $\overline{r}^{\mathbb{E}}$ and $\overline{r}_{\tanh}^{\mathbb{H}}$, respectively. Then $\overline{r}^{\mathbb{E}}$ and $\overline{r}_{\tanh}^{\mathbb{H}}$ are unique metrics in $\mathbb{R}_{>0}^N$ such that $B_{\alpha,i}^{\mathbb{E}}(\overline{r}^{\mathbb{E}})=B_{\alpha,i,\tanh}^{\mathbb{H}}(\overline{r}_{\tanh}^{\mathbb{H}})=\overline{R}\). Moreover, the solutions to the corresponding flows (\ref{pre flow E}) (\ref{pre tanh flow H}) exist for all time and converge exponentially fast to these metrics.

\end{theorem}
\begin{proof}
For $\overline{R}\in C(V)$, we can introduce the following modified branched $\alpha$-potentials, namely the modified $\alpha$-potential
\begin{equation}\label{pre potential E}
    \overline{F}^{\mathbb{E}}(u^{\mathbb{E}})=\int_{u_*^{\mathbb{E}}}^{u^{\mathbb{E}}}\sum_{i = 1}^N (K_i+2\pi\beta_i-\overline{R}_i r_i^{\alpha})\mathrm{d}u_i^{\mathbb{E}},u^{\mathbb{E}}\in U
\end{equation} in $\mathbb{E}^2$, and the modified $\alpha$-$\tanh$-potential
\begin{equation}\label{pre tanh-potential H}
    \overline{F}_{\tanh}^{\mathbb{H}}(u^{\mathbb{H}})=\int_{u_*^{\mathbb{H}}}^{u^{\mathbb{H}}}\sum_{i = 1}^N (K_i+2\pi\beta_i-\overline{R}_i\tanh^{\alpha}\frac{r_i}{2})\mathrm{d}u_i^{\mathbb{H}},u^{\mathbb{H}}\in\mathbb{R}_{<0}^N
\end{equation} in $\mathbb{H}^2$.

It is easy to check that modified branched $\alpha$-potentials
(\ref{pre potential E}) and (\ref{pre tanh-potential H}) are both well-defined. Furthermore, by direct calculation, we have
\[\begin{aligned}
\text{Hess}_{u^{\mathbb{E}}}\overline{F}^{\mathbb{E}}=J^{\mathbb{E}}-\Sigma^{\frac{\alpha}{2}}\begin{pmatrix}
\alpha\overline{R}_1 & & \\
& \ddots & \\
& & \alpha\overline{R}_N
\end{pmatrix}\Sigma^{\frac{\alpha}{2}},\end{aligned}
\] and
\[\begin{aligned}
\text{Hess}_{u^{\mathbb{H}}}\overline{F}^{\mathbb{H}}_{\tanh}=J^{\mathbb{H}}-\Pi^{\frac{\alpha}{2}}\begin{pmatrix}
\alpha\overline{R}_1 & & \\
& \ddots & \\
& & \alpha\overline{R}_N
\end{pmatrix}\Pi^{\frac{\alpha}{2}}.
\end{aligned}\]

Since $\alpha\overline{R}_i\leq0$ for $i = 1,\ldots,N$ and not identically zero, $\text{Hess}_{u^{\mathbb{E}}}\overline{F}^{\mathbb{E}}
$ and $\text{Hess}_{u^{\mathbb{H}}}\overline{F}_{\tanh}^{\mathbb{H}}
$ are both positive definite. The uniqueness of the critical points of $\overline{F}^{\mathbb{E}}$ and $\overline{F}_{\tanh}^{\mathbb{H}}$ follows by Proposition \ref{injective}. Consequently, $\overline{r}^{\mathbb{E}}$ and $\overline{r}_{\tanh}^{\mathbb{H}}$ are unique zero points of $\nabla_{u^{\mathbb{E}}}\overline{F}^{\mathbb{E}}$ and $\nabla_{u^{\mathbb{H}}}\overline{F}_{\tanh}^{\mathbb{H}}$ respectively, which implies that $\overline{r}^{\mathbb{E}}$ and $\overline{r}_{\tanh}^{\mathbb{H}}$ are unique circle packing metrics in $\mathbb{R}_{>0}^N$ such that their $\alpha$-curvature in $\mathbb{E}^2$, and branched $\alpha$-$\tanh$-curvature in $\mathbb{H}^2$ is $\overline{R}$, respectively. Also, by Lemma \ref{critical point to infty}, we know that $\overline{F}^{\mathbb{E}}$ and $\overline{F}_{\tanh}^{\mathbb{H}}$ are both proper and $\lim_{\|u^{\mathbb{E}}\|\to+\infty}\overline{F}^{\mathbb{E}}(u^{\mathbb{E}})=\lim_{\|u^{\mathbb{H}}\|\to+\infty}\overline{F}_{\tanh}^{\mathbb{H}}(u^{\mathbb{H}})=+\infty$. Furthermore, we have $$\frac{\mathrm{d}\overline{F}^{\mathbb{E}}(u^{\mathbb{E}}(t))}{\mathrm{d} t}=-\sum_{i = 1}^N (K_i+2\pi\beta_i-\overline{R}_i r_i^{\alpha})^2\leq 0,$$ $$\frac{\mathrm{d}\overline{F}_{\tanh}^{\mathbb{H}}(u^{\mathbb{H}}(t))}{\mathrm{d} t}=-\sum_{i = 1}^N (K_i+2\pi\beta_i-\overline{R}_i\tanh^{\alpha}\frac{r_i}{2})^2\leq 0.$$
 From the properties mentioned above, the solutions to \eqref{pre flow E} and \eqref{pre tanh flow H} are precompact in $U$ and $\mathbb{R}_{<0}^N$, respectively. Thus we conclude that for each branched $\alpha$-flow, $r(t)$ is bounded below and away from $0$ in both $\mathbb{E}^2$ and $\mathbb{H}^2$. The proof of the existence of the upper bound of $r(t)$ is nearly identical to the arguments of Proposition \ref{long time existance}. Thus, we deduce the long-time existence of the solutions to \eqref{pre flow E} and \eqref{pre tanh flow H}. The remaining proof is similar to that of Theorem \ref{main thm}, which we also omit here.

\end{proof}

The lemma invoked in the preceding theorem asserts:
\begin{lemma}\label{critical point to infty}
   Given a function $F\in C^2(\mathbb{R}^N)$ or $F\in C^2(\mathbb{R}_{<0}^N)$  with $\text{Hess}(F)>0$. If $F$ has a critical point, then $\lim_{\|x\|\to\infty}F(x)=+\infty$. Moreover, $F$ is proper.
\end{lemma}
\begin{proof}
We first prove the condition $F\in C^2(\mathbb{R}^N)$. Without loss of generality, we can assume $\nabla F((-1,\cdots,-1))=0$. For any $\omega\in\mathbb{S}^{N - 1}$, we consider the ray $(-1,\cdots,-1)+\{t\omega\}_{0 \le t < \infty}\subset\mathbb{R}^N$. Let $f_{\omega}(t)=F((-1,\cdots,-1)+t\omega)$, then $f_{\omega}'(0)=0$ and $f_{\omega}''(t)=\omega^T\text{Hess}F\omega>0$, then, as in the case of dimension $N = 1$, there exists constant $a_{\omega}=f_{\omega}'(\frac{1}{2})=\nabla F((-1,\cdots,-1)+\frac{1}{2}\omega)\cdot\omega>0$ such that
\[F((-1,\cdots,-1)+t\omega)=f_{\omega}(t)\geq a_{\omega}t - \frac{1}{2}a_{\omega}+F((-1,\cdots,-1)+\frac{\omega}{2}),\quad\text{for }t \geq \frac{1}{2}.\]

By the arguments above, we have proved $\nabla F((-1,\cdots,-1)+\frac{1}{2}\omega)\cdot\omega>0$ for any $\omega\in\mathbb{S}^{N - 1}$, which implies
\[D:=\inf_{\omega\in\mathbb{S}^{N - 1}}\nabla F((-1,\cdots,-1)+\frac{1}{2}\omega)\cdot\omega>0\]
by the compactness of $\mathbb{S}^{N - 1}$ and that $F$ is $C^2$. Let $B:=\min_{\omega\in\mathbb{S}^{N - 1}}F((-1,\cdots,-1)+\frac{\omega}{2})$, then we have
\[F(x)\geq D\|x-(-1,\cdots,-1)\| - \frac{1}{2}D + B,\quad\text{for }\|x-(-1,\cdots,-1)\|\geq\frac{1}{2}.\]

It implies $\lim_{\|x\|\to\infty}F(x)=+\infty$ and $F$ is proper. The proof for the case $F \in C^2(\mathbb{R}_{<0}^N)$ is similar and only requires restricting $\omega$ to $\mathbb{R}_{\leq 0}^N \cap \mathbb{S}^{N-1}$, so we omit it here.
\end{proof}

\begin{definition}
   A branched combinatorial Gaussian curvature associated with the area element $A_i$ is defined as
\begin{equation}
    R_i=\frac{K_i+2\pi\beta_i}{A_i}.
\end{equation}
In the following, we call it branched $A$-curvature for short.
\end{definition}
\begin{definition}
A combinatorial Ricci flow with respect to area element $A_i$ is defined as

\begin{equation}\label{A flow}
\frac{\mathrm{d}u_i}{\mathrm{d}t}=\gamma_i\cdot(\overline{R}_i - R_i),
\end{equation}where $\gamma_i > 0$ is a smooth function defined on $\mathbb{R}_{>0}^N$, $\overline{R}$ is a prescribed function. $u_i=\operatorname{ln} r_i$ in $\mathbb{E}^2$ and $u_i = \ln\tanh\frac{r_i}{2}$ in $\mathbb{H}^2$ respectively. In the following, we call it branched $A$-curvature flow or branched $A$-flow for short.\end{definition} Let
\begin{equation}
    A_i^{\mathbb{E}} = \pi r_i^{\alpha}
\end{equation} in $\mathbb{E}^2$ and
\begin{equation}
    A_i^{\mathbb{H}} = 4\pi\sinh^{\alpha}\frac{r_i}{2}
\end{equation}
in $\mathbb{H}^2$ respectively. In fact, when $\alpha=2$, then $4\pi\sinh^2\frac{r_i}{2}$ and $\pi r_i^2$  are precisely the areas of the hyperbolic and Euclidean disks with radius $r_i$ packed at the vertex $v_i$, respectively.

Define a branched $A$-curvature in $\mathbb{E}^2$ as
\begin{equation}
    R_{H,A,i}^{\mathbb{E}}=\frac{K_i+2\pi\beta_i}{A_i^{\mathbb{E}}},
\end{equation} and a branched $A$-curvature in $\mathbb{H}^2$ as
\begin{equation}
    R_{H,A,i}^{\mathbb{H}}=\frac{K_i+2\pi\beta_i}{A_i^{\mathbb{H}}}.
\end{equation}
Then the corresponding branched $\alpha$-flows are defined to be
\begin{equation}\label{area flow E}
\frac{\mathrm{d}u_i^{\mathbb{E}}}{\mathrm{d}t}=\overline{R}_{i}-R_{H,A,i}^{\mathbb{E}},
\end{equation} in $\mathbb{E}^2$ and
\begin{equation}\label{area flow H}
\frac{\mathrm{d}u_i^{\mathbb{H}}}{\mathrm{d}t}=\overline{R}_{i}-R_{H,A,i}^{\mathbb{H}}.
\end{equation} in $\mathbb{H}^2$, respectively.

We define \(\overline{R}\) as \emph{$A$-admissible} in \(\mathbb{E}^2\) if there exists a circle packing metric \(\overline{r}^{\mathbb{E}}\) such that its branched $A$-curvature in \(\mathbb{E}^2\) equals \(\overline{R}\). Analogously, \(\overline{R}\) is called \emph{A-admissible} in \(\mathbb{H}^2\) when there exists a circle packing metric \(\overline{r}^{\mathbb{H}}\) with its branched $A$-curvature in \(\mathbb{H}^2\) equal to \(\overline{R}\).
\begin{theorem}\label{thm area}
    Suppose $A_i^{\mathbb{E}}=\pi r_i^{\alpha}$ in $\mathbb{E}^2$ and $A_i^{\mathbb{H}} = 4\pi\sinh^{\alpha}\frac{r_i}{2}$ in $\mathbb{H}^2$ respectively, where $\alpha$ is a non-negative real number. Suppose \(\overline{R}\in C(V)\) is a function with \(\overline{R}_i \leq 0\) for all \(i\in V\), and $\{\alpha\overline{R}_i\}_{1\leq i\leq N}$ are not identically zero. If $\overline{R}$ is $A$-admissible in $\mathbb{E}^2$ with $R_{H,A,i}^{\mathbb{E}}(\overline{r}^\mathbb{{E}}) = \overline{R}$ and similarly $A$-admissible in $\mathbb{H}^2$ with $R_{H,A,i}^{\mathbb{H}}(\overline{r}^\mathbb{{H}}) = \overline{R}$, then $\overline{r}^\mathbb{{E}}$ and $\overline{r}^\mathbb{{H}}$ are the unique circle packing metrics with branched $A$-curvature $\overline{R}$ in $\mathbb{E}^2$ and $\mathbb{H}^2$, respectively. Furthermore, $\overline{R}$ is the branched $A$-curvature of $\overline{r}^{\mathbb{E}}$, and similarly of $\overline{r}^{\mathbb{H}}$ if and only if the solutions to \eqref{area flow E} and \eqref{area flow H} exist for all time and converge exponentially fast to $\overline{r}^{\mathbb{E}}$ and $\overline{r}^{\mathbb{H}}$ in $\mathbb{E}^2$ and $\mathbb{H}^2$, respectively.
\end{theorem}
\begin{proof}
To prove the uniqueness of $\overline{r}^{\mathbb{E}}$ and $\overline{r}^{\mathbb{H}}$ with the branched $A$-curvature $\overline{R}$, we introduce the following branched combinatorial Ricci $A$-potentials (branched $A$-potentials for short).
\begin{equation}
F_A^{\mathbb{E}}(u^{\mathbb{E}})=\int_{\overline{u}^{\mathbb{E}}}^{u^{\mathbb{E}}}\sum_{i = 1}^{N}(K_i+2\pi\beta_i - \overline{R}_iA_i^{\mathbb{E}})\mathrm{d}u_i^{\mathbb{E}}
\end{equation} in $\mathbb{E}^2$, and
\begin{equation}\label{area potential H}
F_A^{\mathbb{H}}(u^{\mathbb{H}})=\int_{\overline{u}^{\mathbb{H}}}^{u^{\mathbb{H}}}\sum_{i = 1}^{N}(K_i+2\pi\beta_i - \overline{R}_iA_i^{\mathbb{H}})\mathrm{d}u_i^{\mathbb{H}}
\end{equation} in $\mathbb{H}^2$.

It is easy to check that the functionals $F_A^{\mathbb{E}}$ and $F_A^{\mathbb{H}}$ are both well-defined. By direct calculations, we have
\begin{equation}\begin{aligned}
    \text{Hess}_{u^\mathbb{E}}F_A^{\mathbb{E}}= J^{\mathbb{E}}-2\pi\alpha\begin{pmatrix}
\overline{R}_1 r_1^{\alpha} & & \\
& \ddots & \\
& & \overline{R}_N r_N^{\alpha}
\end{pmatrix}\end{aligned}
\end{equation} in $\mathbb{E}^2$ and
\begin{equation}\begin{aligned}
    \text{Hess}_{u^\mathbb{H}}F_A^{\mathbb{H}} = J^{\mathbb{H}}-4\pi\alpha\begin{pmatrix}
\overline{R}_1 \sinh^{\alpha}(\frac{r_1}{2})\cosh^2 (\frac{r_1}{2}) & & \\
& \ddots & \\
& & \overline{R}_N \sinh^{\alpha}(\frac{r_N}{2})\cosh^2 (\frac{r_N}{2})
\end{pmatrix}\end{aligned}
\end{equation} in $\mathbb{H}^2$,
which are both positive definite by $\alpha\overline{R}_i \leq 0$ and do not contain the identically zero cases. Note that $r$ is a circle packing metric with the branched $A$-curvature $\overline{R}$ in $\mathbb{E}^2$ if and only if $\nabla_{u^{\mathbb{E}}}F_A(r) = 0$. The analogous condition holds for $\mathbb{H}^2$ as well. Thus, the uniqueness of metrics with the curvature $\overline{R}$ follows from Lemma \ref{injective}.

For the “only if” part of the equivalence between the $A$-admissibility of $\overline{R}$ and the convergence of the branched $A$-flow (\ref{area flow E}) and (\ref{area flow H}), suppose the solution $r(t)\to\overline{r}^{\mathbb{E}}$ in $\mathbb{E}^2$ and $r(t)\to\overline{r}^{\mathbb{H}}$ in $\mathbb{H}^2$ as \(t\to+\infty\). In the case of Euclidean background geometry, we have $u^{\mathbb{E}}(t)\to\overline{u}^{\mathbb{E}}\), which implies that there exists a sequence $\xi_n\to+\infty$ such that
\[
u_i^{\mathbb{E}}(n + 1)-u_i^{\mathbb{E}}(n)={u_i^{\mathbb{E}}}^\prime(\xi_n)\to0
\]
as $n\to+\infty$. As $r(t)\to\overline{r}^{\mathbb{E}}$, we have $ R_{H,A,i}^{\mathbb{E}}(r(t))\to R_{H,A,i}^{\mathbb{E}}(\overline{r}^{\mathbb{E}})$. Thus, $ R_{H,A,i}^{\mathbb{E}}(\overline{r}^{\mathbb{E}})=\overline{R}_i$, that is, $\overline{r}^{\mathbb{E}}$ is a circle packing metric with the curvature $\overline{R}$. The proof for the case of hyperbolic background geometry is the same as above.

For the "if" part, we first prove that the solution of \eqref{area flow E} and \eqref{area flow H} stays in a compact set of $\mathbb{R}_{>0}^N$.
It is straightforward to verify that $F_A^{\mathbb{E}}(u^{\mathbb{E}}) \geq F_A^{\mathbb{E}}(\overline{u}^{\mathbb{E}}) = 0$ and $F_A^{\mathbb{H}}(u^{\mathbb{H}}) \geq F_A^{\mathbb{H}}(\overline{u}^{\mathbb{H}}) = 0$. Furthermore, $\overline{u}^{\mathbb{E}}$ is the unique minimum and critical point of $F_A^{\mathbb{E}}$, while $\overline{u}^{\mathbb{H}}$ serves as the unique minimum and critical point of $F_A^{\mathbb{H}}$. Moreover, $\lim_{\|u^{\mathbb{E}}\|\to+\infty}F_A^{\mathbb{E}}(u^{\mathbb{E}})=\lim_{\|u^{\mathbb{H}}\|\to+\infty}F_A^{\mathbb{H}}(u^{\mathbb{H}})=+\infty$
 by Lemma \ref{critical point to infty}. By direct calculations, we have
\begin{equation}
 \begin{aligned}
\frac{\mathrm{d}F_A^{\mathbb{E}}(u^{\mathbb{E}}(t))}{\mathrm{d}t}=\nabla F_A^{\mathbb{E}}\cdot\frac{\mathrm{d}u^{\mathbb{E}}}{\mathrm{d}t}&=\sum_{i}\left(K_i+2\pi\beta_i - 2\pi\overline{R}_ir_i^{\alpha}\right)\cdot(\overline{R}_i - R_{H,A,i}^{\mathbb{E}})\\
&=- 2\pi\sum_{i}r_i^{\alpha}(R_{H,A,i}^{\mathbb{E}} - \overline{R}_i)^2\leq0,
\end{aligned}
\end{equation}
 and
 \begin{equation}
  \begin{aligned}
\frac{\mathrm{d}F_A^{\mathbb{H}}(u^{\mathbb{H}}(t))}{\mathrm{d}t}=\nabla F_A^{\mathbb{H}}\cdot\frac{\mathrm{d}u^{\mathbb{H}}}{\mathrm{d}t}&=\sum_{i}\left(K_i+2\pi\beta_i - 4\pi\overline{R}_i\sinh^{\alpha}\frac{r_i}{2}\right)\cdot(\overline{R}_i - R_{H,A,i}^{\mathbb{H}})\\
&=- 4\pi\sum_{i}\sinh^{\alpha}\frac{r_i}{2}(R_{H,A,i}^{\mathbb{H}} - \overline{R}_i)^2\leq0.
\end{aligned}
\end{equation}
Then $F_A^{\mathbb{E}}$ and $F_A^{\mathbb{H}}$ are both decreasing. Using  $\lim_{\|u^{\mathbb{E}}\|\to+\infty}F_A^{\mathbb{E}}(u^{\mathbb{E}})=\lim_{\|u^{\mathbb{H}}\|\to+\infty}F_A^{\mathbb{H}}(u^{\mathbb{H}})=+\infty$ , we know $\{u^{\mathbb{E}}(t)\}$ and $\{u^{\mathbb{H}}(t)\}$ are both bounded, thus $r(t)$ has a positive lower bound in both $\mathbb{E}^2$ and $\mathbb{H}^2$. $r(t)$ is also uniformly bounded from above in both $\mathbb{E}^2$ and $\mathbb{H}^2$ by nearly identical proof of proposition \ref{long time existance}. Therefore we obtain the long time existence of the solution of \eqref{area flow E} and \eqref{area flow H}.

Set $\overline{\Gamma}^{\mathbb{E}}(u^{\mathbb{E}})=\overline{R}-R_{H,A}^{\mathbb{E}}$ in $\mathbb{E}^2$ and $\overline{\Gamma}^{\mathbb{H}}(u^{\mathbb{H}}) = \overline{R}-R_{H,A}^{\mathbb{H}}$ respectively, here $R_{H,A}^{\mathbb{E}}$ denotes $(R_{H,A,1}^{\mathbb{E}},\ldots,R_{H,A,N}^{\mathbb{E}})^T$ and $R_{H,A}^{\mathbb{H}}$ denotes $(R_{H,A,1}^{\mathbb{H}},\ldots,R_{H,A,N}^{\mathbb{H}})^T$ respectively.

Then we have
\begin{equation}
D\overline{\Gamma}^{\mathbb{E}}|_{\overline{u}^{\mathbb{E}}}=-\frac{1}{2\pi} \operatorname{diag}\{r_1^{\alpha},\ldots,r_N^{\alpha}\}^{-1}\cdot \text{Hess}_{u^{\mathbb{E}}}F_A^{\mathbb{E}}
\end{equation}
in $\mathbb{E}^2$ and
\begin{equation}
D\overline{\Gamma}^{\mathbb{H}}|_{\overline{u}^{\mathbb{H}}}=-\frac{1}{4\pi}\operatorname{diag}\{\sinh^{\alpha}\frac{r_1}{2},\ldots,\sinh^{\alpha}\frac{r_N}{2}\}^{-1}\cdot\text{Hess}_{u^{\mathbb{H}}}F_A^{\mathbb{H}}
\end{equation} in $\mathbb{H}^2$.

Since $D\overline{\Gamma}^{\mathbb{E}}|_{\overline{u}^{\mathbb{E}}}\sim -\frac{1}{2\pi}\operatorname{diag}\{r_1^{\alpha},\ldots,r_N^{\alpha}\}^{-\frac{1}{2}}\cdot \text{Hess}_{u^{\mathbb{E}}}F_A^{\mathbb{E}}\cdot \operatorname{diag}\{r_1^{\alpha},\ldots,r_N^{\alpha}\}^{-\frac{1}{2}}$ in $\mathbb{E}^2$ and $D\overline{\Gamma}^{\mathbb{H}}|_{\overline{u}^{\mathbb{H}}}\sim -\frac{1}{4\pi}\operatorname{diag}\{\sinh^{\alpha}\frac{r_1}{2},\ldots,\sinh^{\alpha}\frac{r_N}{2}\}^{-\frac{1}{2}}\cdot \text{Hess}_{u^{\mathbb{H}}}F_A^{\mathbb{H}}\cdot \operatorname{diag}\{\sinh^{\alpha}\frac{r_1}{2},\ldots,\sinh^{\alpha}\frac{r_N}{2}\}^{-\frac{1}{2}}$ in $\mathbb{H}^2$, so they both have $N$ negative eigenvalues.
Therefore, by ODE theory, for any initial metrics $u^{\mathbb{E}}(0)$ and $u^{\mathbb{H}}(0)$, the solutions $u^{\mathbb{E}}(t)$ and $u^{\mathbb{H}}(t)$ converge exponentially fast to $\overline{u}^{\mathbb{E}}$ and $\overline{u}^{\mathbb{H}}$, respectively.
\end{proof}

\begin{theorem}
Given a branched weighted triangulated closed surface \((M, \mathcal{T}, \Phi, H)\), a non-negative real number \(\alpha\), and a prescribed function \(\overline{R} \in C(V)\) such that \(\overline{R}_i \leq 0\) for all vertices \(i \in V\) and $\{\alpha\overline{R}_i\}_{1\leq i\leq N}$ are not identically zero, we establish that \(\overline{R}\) is:
\begin{enumerate}
    \item admissible and A-admissible in \(\mathbb{E}^2\),
    \item tanh-admissible and A-admissible in \(\mathbb{H}^2\),
\end{enumerate}
each via a unique metric determined by the respective geometric structure.
\end{theorem}
\begin{proof}
Notice that $\overline{R}$ is a bounded function since it is defined on a finite discrete vertex set $V$. As $r_i\to 0$, we have $\overline{R}_ir_i^{\alpha}\to 0^-$, $\overline{R}_i\tanh^\alpha\frac{r_i}{2}\to 0^-$, and $\overline{R}_i\sinh^\alpha\frac{r_i}{2}\to 0^-$. Combined with the fact that \(\overline{R}_i \leq 0\), we can adapt the proofs of Theorems \ref{thm critical point E} and \ref{thm critical point H} mutatis mutandis (with necessary modifications) and complete the proof, so we omit it here.
\end{proof}
\section{Some Problems}\label{sec 7}
\hspace{14pt}
In this section, we formulate several open problems yet to be solved. In Remark \(\ref{reason not sinh}\), we provided justification for excluding the equation \(\eqref{wrong flow}\) as a branched \(\alpha\)-flow in \(\mathbb{H}^2\), citing the ill-posedness of its potential \(\eqref{wrong potential}\). A question arises: is there an alternative framework, independent of potential-based methods, to tackle the challenges that follow?

The first problem is as follows:
\begin{problem}
Does the branched \(\alpha\)-flow \(\eqref{wrong flow}\) converge exponentially in \(\mathbb{H}^2\)?
\end{problem}

Secondly, the unsolved problem is on the maximum principle of \eqref{main flow E} and \eqref{main flow H}. In the following discussions, we assume $\chi(M)\not=0$.
\begin{lemma}[\cite{1980The}\label{triangle}, Lemma 13.7.3]
 In both two background geometries \(\mathbb{E}^2\) and \(\mathbb{H}^2\), we have
 \begin{enumerate}
     \item \(\partial\theta_i/\partial r_i < 0\),
     \item \(\partial\theta_i/\partial r_j> 0\) for \(i\neq j\),
     \item \(\partial(\theta_i + \theta_j+\theta_k)/\partial r_i < 0\) in \(\mathbb{H}^2\), and \(\partial(\theta_i + \theta_j+\theta_k)/\partial r_i = 0\) in \(\mathbb{E}^2\).
 \end{enumerate}
Furthermore, the partial derivatives \(\partial\theta_n/\partial r_m\) are elementary functions in \(r_i\), \(r_j\) and \(r_k\) where \(n, m\in\{i, j, k\}\).
\end{lemma}

\begin{lemma}[\cite{2002Combinatorial}, A3]\label{ij}
In $\mathbb{E}^2$ and $\mathbb{H}^2$, we have
\[
\frac{\partial \theta_i}{\partial u_j^{\mathbb{E}}}=\frac{\partial \theta_j}{\partial u_i^{\mathbb{E}}},
\frac{\partial \theta_i}{\partial u_j^{\mathbb{H}}}=\frac{\partial \theta_j}{\partial u_i^{\mathbb{H}}}.
\]
\end{lemma}
\begin{proposition}\label{elementary}
Under branched $\alpha$-flows \eqref{main flow E} and \eqref{main flow H}, $K_i$ evolves as
\begin{equation}\label{elementary E}
\frac{\mathrm{d} K_i}{\mathrm{d}t}=\sum_{v_j\sim v_i}C^{\mathbb{E}}_{ij}[(K_j+2\pi\beta_j - K_i-2\pi\beta_i)-s_{\alpha}^{\mathbb{E}}(r_j^{\alpha}-r_i^\alpha)]
\end{equation}
in $\mathbb{E}^2$, and
\begin{equation}\label{elementary H}
\frac{\mathrm{d}K_i}{\mathrm{d}t}=\sum_{v_j\sim v_i}C^{\mathbb{H}}_{ij}[(K_j+2\pi\beta_j - K_i-2\pi\beta_i)-s_{\alpha}^{\mathbb{H}}(\tanh^{\alpha}\frac{r_j}{2}-\tanh^{\alpha}\frac{r_i}{2})]-S_i(K_i+2\pi\beta_i-s_{\alpha}^{\mathbb{H}}\tanh^{\alpha}\frac{r_i}{2})
\end{equation}
in $\mathbb{H}^2$, respectively. Here, the summation $\sum_{v_j\sim v_i}$ runs over all vertices $v_j$ that are adjacent to vertex $v_i$. The functions $C_{ij}^{\mathbb{E}}$, $C_{ij}^{\mathbb{H}}$ and $S_i$ are positive elementary functions of the radii $r_1,\ldots,r_N$. Additionally, $C_{ij}^{\mathbb{E}}$ and $C_{ij}^{\mathbb{H}}$ are both symmetric.
\end{proposition}
\begin{proof}
For convenience, we first define functions $c(u)$, $c(s_{\alpha})$ and $c(r)$ in different geometries:
if the background geometry is $\mathbb{E}^2$, then $c(u)$ is $u^{\mathbb{E}}$, $c(s_{\alpha})$ is $s_{\alpha}^{\mathbb{E}}$, and $c(r)$ is $r^{\alpha}$; if the background geometry is $\mathbb{H}^2$, then $c(u)$ is $u^{\mathbb{H}}$, and $c(s_{\alpha})$ is $s_{\alpha}^{\mathbb{H}}$, $c(r)$ is $\tanh^{\alpha}\frac{r_i}{2}$.

For the triangle $\Delta_{ijk}$, by the chain rule, we have
\begin{equation}
\begin{aligned}
d\theta_i/dt&=\partial\theta_i/\partial c(u_i)c(u_i)'+\partial\theta_i/\partial c(u_j)c(u_j)'+\partial\theta_i/\partial c(u_k)c(u_k)'\\
&=\partial\theta_i/\partial c(u_j)[c(s_{\alpha})(c(r_j)-c(r_i))-(K_j+2\pi\beta_j-K_i-2\pi\beta_i)]\\&+\partial\theta_i/\partial c(u_k)[c(s_{\alpha})(c(r_k)-c(r_i))-(K_k+2\pi\beta_k-K_i-2\pi\beta_i)]\\&+D_i(c(s_{\alpha})c(r_i)-K_i-2\pi\beta_i).
\end{aligned}
\end{equation}
Here $D_i=\partial\theta_i/\partial c(u_i)+\partial\theta_i/\partial c(u_j)+\partial\theta_i/\partial c(u_k)$. By Lemma \ref{ij}, $D_i$ can be written as $\partial\theta_i/\partial c(u_i)+\partial\theta_j/\partial c(u_i)+\partial\theta_k/\partial c(u_i)$. Set $A_{ii}=-D_i$ in $\mathbb{H}^2$ and $A_{ii}=1$ in $\mathbb{E}^2$, then by Lemma \ref{triangle}, we get that $A_{ii}>0$. Let $A_{mn}=\partial\theta_m/\partial c(u_n)$ for $n\not=m$. Then by Lemma \ref{triangle}, it is positive and an elementary function. Moreover, we get $A_{mn}=A_{nm},m\not =n$ by Lemma \ref{ij}. Since $dK_i/dt = -\sum_{j,k} d\theta_i^{jk}/dt$, Thus the result follows.
\end{proof}

Calculate
\begin{equation}\label{elementary2 E}
\begin{aligned}
\frac{\mathrm{d} (K_i+2\pi\beta_i)-s_{\alpha}^{\mathbb{E}}r_i^{\alpha}}{\mathrm{d}t}&=\sum_{v_j\sim v_i}C^{\mathbb{E}}_{ij}[(K_j+2\pi\beta_j - K_i-2\pi\beta_i)-s_{\alpha}^{\mathbb{E}}(r_j^{\alpha}-r_i^\alpha)]\\ &+\frac{\alpha (s_{\alpha}^{\mathbb{E}})^2r_i^{\alpha}}{2\pi\chi(M)}\sum_{1\leq l\leq N}[(K_i+2\pi\beta_i - K_l-2\pi\beta_l)-s_{\alpha}^{\mathbb{E}}(r_i^{\alpha}-r_l^\alpha)]r_l^{\alpha}
\end{aligned}
\end{equation} and
\begin{equation}\label{elementary2 H}
\begin{aligned}
\frac{\mathrm{d} (K_i+2\pi\beta_i)-s_{\alpha}^{\mathbb{H}}\tanh^{\alpha}\frac{r_i}{2}}{\mathrm{d}t}&=\sum_{v_j\sim v_i}C^{\mathbb{H}}_{ij}[(K_j+2\pi\beta_j - K_i-2\pi\beta_i)-s_{\alpha}^{\mathbb{H}}(\tanh^{\alpha}\frac{r_j}{2}-\tanh^{\alpha}\frac{r_i}{2})]\\ &-S_i(K_i+2\pi\beta_i-s_{\alpha}^{\mathbb{H}}\tanh^{\alpha}\frac{r_i}{2})+\frac{\alpha (s_{\alpha}^{\mathbb{H}})^2\tanh^{\alpha}\frac{r_i}{2}}{2\pi\chi(M)}\\&\sum_{1\leq l\leq N}[(K_i+2\pi\beta_i - K_l-2\pi\beta_l)-s_{\alpha}^{\mathbb{H}}(\tanh^{\alpha}\frac{r_i}{2}-\tanh^{\alpha}\frac{r_l}{2})]\tanh^{\alpha}\frac{r_l}{2}.
\end{aligned}
\end{equation}

Let $u^{\mathbb{E}}(t) = (u^{\mathbb{E}}_1(t),\ldots,u^{\mathbb{E}}_N(t))$ be a solution to
\eqref{main flow E} and $u^{\mathbb{H}}(t) = (u^{\mathbb{H}}_1(t),\ldots,u^{\mathbb{H}}_N(t))$ be a solution to
\eqref{main flow H}.
Set
\begin{equation}\label{max E}
G_v^{\mathbb{E}}(t)=\max_{1\leq i\leq N}\left\{K_1(t)+2\beta_1\pi-s_{\alpha}^{\mathbb{E}}r_1^{\alpha},K_2(t)+2\beta_2\pi-s_{\alpha}^{\mathbb{E}}r_2^{\alpha},\ldots,K_N(t)+2\beta_N\pi-s_{\alpha}^{\mathbb{E}}r_N^{\alpha}\right\},
\end{equation}
\begin{equation}\label{max H}
\begin{aligned}
G_p^{\mathbb{H}}(t)&=\max\bigg(\max_{1\leq i\leq N}\bigg\{K_1(t)+2\beta_1\pi-s_{\alpha}^{\mathbb{H}}\tanh^{\alpha}\frac{r_1}{2},K_2(t)+2\beta_2\pi-s_{\alpha}^{\mathbb{H}}\tanh^{\alpha}\frac{r_2}{2},\ldots,K_N(t)\\& +2\beta_N\pi-s_{\alpha}^{\mathbb{H}}\tanh^{\alpha}\frac{r_N}{2}\bigg\},0\bigg),
\end{aligned}
\end{equation}
\begin{equation}\label{min E}
G_w^{\mathbb{E}}(t)=\min_{1\leq i\leq N}\left\{K_1(t)+2\beta_1\pi-s_{\alpha}^{\mathbb{E}}r_1^{\alpha},K_2(t)+2\beta_2\pi-s_{\alpha}^{\mathbb{E}}r_2^{\alpha},\ldots,K_N(t)+2\beta_N\pi-s_{\alpha}^{\mathbb{E}}r_N^{\alpha}\right\},
\end{equation}
\begin{equation}\label{min H}
\begin{aligned}
G_q^{\mathbb{H}}(t)&=\min\bigg(\min_{1\leq i\leq N}\bigg\{K_1(t)+2\beta_1\pi-s_{\alpha}^{\mathbb{H}}\tanh^{\alpha}\frac{r_1}{2},K_2(t)+2\beta_2\pi-s_{\alpha}^{\mathbb{H}}\tanh^{\alpha}\frac{r_2}{2},\ldots,K_N(t)\\&+2\beta_N\pi-s_{\alpha}^{\mathbb{H}}\tanh^{\alpha}\frac{r_N}{2}\bigg\},0\bigg).
\end{aligned}
\end{equation}

Substitute the functions in \eqref{max E} and \eqref{min E} into \eqref{elementary2 E}, and those in \eqref{max H} and \eqref{min H} into \eqref{elementary2 H}, respectively. However, a problem arises. For convenience, we take $G_v^{\mathbb{E}}(t)$ as an example. From the right-hand side of \eqref{max E}, we observe that the first term is non-positive, while the second term is non-negative. Hence the maximum principle cannot be directly applied. The second question is:
\begin{problem}
For \(\alpha\)-flows \eqref{main flow E} and \eqref{main flow H}, what conditions and restrictions are required to guarantee the corresponding maximum principles?
\end{problem}

\noindent{\bf Acknowledgments.}
The fourth author is supported by NSFC (No. 12171480), Natural Science Foundation of Hunan Province (No. 2022JJ10059) and Scientific Research Program of NUDT (No. JS2023-01). All authors would like to thank professor Huabin Ge for many useful conversations.

\noindent Wenjun Li,  liwenjun22@nudt.edu.cn\\
\emph{Department of Mathematics, National University of Defense Technology, Changsha 410073, P. R. China.}\\

\noindent Rongyuan Liu, 3102389857@qq.com\\
\emph{Department of Mathematics, National University of Defense Technology, Changsha 410073, P. R. China.}\\

\noindent Guohao Chen, m305893667@qq.com\\
\emph{Department of Mathematics, National University of Defense Technology, Changsha 410073, P. R. China.}\\

\noindent Aijin Lin, linaijin@nudt.edu.cn\\
\emph{Department of Mathematics, National University of Defense Technology, Changsha 410073, P. R. China.}\\

\begin{thebibliography}{99}
    \bibitem{1990The} Beardon, A.F. and Stephenson, K., \emph{The Uniformization Theorem for Circle Packings}. Indiana University Mathematics Journal, vol. s4-39, pp. 1383-1425, 1990.

    \bibitem{bowers} Bowers, P. and Stephenson, K., \emph{A branched Andreev-Thurston theorem for circle packings of the spheres}, Proceedings of the London Mathematical Society, vol. s3-73, pp. 185-215, 1996.

    \bibitem{2002Combinatorial} Chow, B. and Luo, F., \emph{Combinatorial Ricci flows on surfaces}. Journal of Differential Geometry, vol. 63, pp. 97-129, 2003.

    \bibitem{DV} de Verdiere, Y.C., \emph{Un principe variationnel pour les empilements de cercles}, Inventiones Mathematicae, vol. s3-104, pp. 655-669, 1991.

    \bibitem{1995Branched} Dubejko, T., \emph{Branched circle packings and discrete Blaschke products}. Transactions of the American Mathematical Society, vol. 347, pp. 4073-4103, 1995.

     \bibitem{gl1} Gao, K. and Lin, A., \emph{Branched combinatorial Calabi flows on surface}, Advances in Applied Mathematics (Chinese), vol. s10-11, pp. 7451-7463, 2022.

     \bibitem{gl2} Gao, K. and Lin, A., \emph{Branched combinatorial p-th Ricci flows on surfaces}. Rend. Circ. Mat. Palermo, vol. s7-72, pp. 3363-3375, 2023.


    \bibitem{2015} Ge, H. and Xu, X., \emph{$\alpha$-curvatures and $\alpha$-flows on low dimensional triangulated manifolds}. Calculus of Variations and Partial Differential Equations, vol. 55, pp. 1-26, 2016.

    \bibitem{2015A} Ge, H. and Xu, X., \emph{A combinatorial Yamabe problem on two and three dimensional manifolds}. Calculus of Variations and Partial Differential Equations, vol. 60, pp. 1-45, 2021.

    \bibitem{2007Variational} Lan, S. and Dai, D., \emph{Variational principles for branched circle patterns}. Nonlinear Analysis-theory Methods and Applications, vol. 67, pp. 498-511, 2007.

    \bibitem{Hamilton1982Three} Hamilton, R.S., \emph{Three Manifolds with Positive Ricci Curvature}. Journal of Differential Geometry, vol. 17, pp. 255-306, 1982.

    \bibitem{1980The} Thurston. W., \emph{Geometry and Topology of 3-Manifolds}. Princeton Lecture Notes, 1976.


\end{thebibliography}
\end{document}